\documentclass[10pt,reqno,a4paper,oneside,11pt]{amsart}%
\usepackage{amsfonts}
\usepackage{amsfonts}
\usepackage{amsfonts}
\usepackage{amsfonts}
\usepackage{amsfonts}
\usepackage{amsfonts}
\usepackage{amsfonts}
\usepackage{amsfonts}
\usepackage{amsfonts}
\usepackage{amsfonts}
\usepackage{amsfonts}
\usepackage{amsfonts}
\usepackage{amsfonts}
\usepackage{amsfonts}
\usepackage{amsfonts}
\usepackage{amsfonts}
\usepackage{amsfonts}
\usepackage{amsfonts}
\usepackage{amsfonts}
\usepackage{amsfonts}
\usepackage{mathrsfs}
\usepackage{mathrsfs}
\usepackage{amsfonts}
\usepackage{amssymb}
\usepackage{amsmath}
\usepackage{amsthm}
\usepackage{graphicx}
\usepackage{color}
\providecommand{\U}[1]{\protect\rule{.1in}{.1in}}
\newtheorem{theorem}{Theorem}[section]

\newtheorem{corollary}[theorem]{Corollary}
\newtheorem{lemma}[theorem]{Lemma}
\newtheorem{proposition}[theorem]{Proposition}

\newtheorem{definition}{Definition}[section]
\newtheorem{remark}{Remark}[section]

\newtheorem{example}{Example}

\theoremstyle{definition}
\theoremstyle{remark}
\numberwithin{equation}{section}

\ifx\pdfoutput\relax\let\pdfoutput=\undefined\fi
\newcount\msipdfoutput
\ifx\pdfoutput\undefined\else
\ifcase\pdfoutput\else
\ifx\paperwidth\undefined\else
\ifdim\paperheight=0pt\relax\else\pdfpageheight\paperheight\fi
\ifdim\paperwidth=0pt\relax\else\pdfpagewidth\paperwidth\fi
\fi\fi\fi
\begin{document}
\pagestyle{myheadings}

\begin{center}
{\Large \textbf{Tensor decompositions and tensor equations over quaternion algebra}}\footnote{This research was supported by the grants from the National Natural
Science Foundation of China (11571220).
\par
{}* Corresponding author.
\par  Email address: hzh19871126@126.com (Zhuo-Heng He); cnavasca@uab.edu (C. Navasca); wqw@t.shu.edu.cn (Q.W. Wang);}

\bigskip

{ \textbf{Zhuo-Heng He$^{a,c}$, Carmeliza Navasca$^{b}$, Qing-Wen Wang$^{c,*}$}}

{\small
\vspace{0.25cm}
 $a.$ Department of Mathematics and Statistics, Auburn University, AL 36849-5310, USA\\

 $b.$ Department of Mathematics, University of Alabama at Birmingham, Birmingham, AL 35294, USA\\

 $c.$ Department of Mathematics, Shanghai University, Shanghai 200444, P. R. China}

\end{center}

\begin{quotation}
\noindent\textbf{Abstract:} In this paper, we investigate and discuss in detail the structures of quaternion tensor SVD, quaternion tensor rank decomposition, and $\eta$-Hermitian quaternion tensor decomposition with the isomorphic group structures and Einstein product. Then we give the expression of the Moore-Penrose inverse of a quaternion tensor by using the quaternion tensor SVD. Moreover, we consider a generalized Sylvester quaternion tensor equation. We give some necessary and sufficient conditions for the existence of a solution to the generalized Sylvester quaternion tensor equation in terms of the Moore-Penrose inverses of the quaternion tensors. We also present the expression of the general solution to this tensor equation when it is solvable. As applications of this generalized Sylvester quaternion tensor equation, we derive some necessary and sufficient conditions for the existences of $\eta$-Hermitian solutions to some quaternion tensor equations. We also provide some numerical examples to illustrate our results.
\end{quotation}

\section{\textbf{Introduction}}

Decompositions of of higher-order tensors have found huge applications in signal processing (\cite{comon01}, \cite{lieven01}, \cite{lieven02}, \cite{lieven03}, \cite{lieven04}, \cite{limlh01}, \cite{muti01}, \cite{navasca02}, \cite{Sidiropoulos01}),  data mining (\cite{nliu001}, \cite{bsavas01}, \cite{jsun01}), genomic signals (\cite{orly04}, \cite{alter01}, \cite{alter02}, \cite{alter03}), computer vision (\cite{shashua01}, \cite{VASILESCU01}, \cite{VASILESCU02}), higher-order statistics (\cite{comon03}, \cite{comon04}, \cite{lieven05}), pattern recognition (\cite{YDKim01}, \cite{bsavas02}), chemometrics (\cite{comon02}, \cite{chemical02}), graph analysis \cite{kolda01}, numerical linear algebra (\cite{lieven06}, \cite{lieven07}, \cite{kolda02}, \cite{kolda03}, \cite{tzhang01}), aerospace engineering \cite{Doostan}, and elsewhere. Kolda and Bader \cite{kolda04} provided an overview of the theoretical developments and applications of tensor decompositions in 2009. There have been many papers discussing tensor decompositions and other fields of tensor theory (\cite{navasca}-\cite{changkungching03}, \cite{comon05}, \cite{lieven06},\cite{weiym01}, \cite{liwen01}, \cite{limlh02}, \cite{limlh03}, \cite{qi01}-\cite{qi05}, \cite{shaojiayu01}-\cite{shaojiayu04}, \cite{yangyuning02}, \cite{yangyuning01}).

Tensor equations are found to be useful in engineering and science. For instance, tensor equations can be used to model some problems in continuum physics and engineering, isotropic and anisotropic elasticity (\cite{lai01}). There have been some papers using different approaches to investigate tensor equation over fields (\cite{navasca}, \cite{chenzhen01}, \cite{weiym02}, \cite{Sun}).

The concept of quaternions was introduced by W.R. Hamilton in 1843 \cite{hamilton}. Quaternion algebra, which is an associative and noncommutative division algebra over the real number field, can be used to computer science, quantum physics, signal and color image processing, and so on (\cite{S. De Leo}, \cite{Took1}, \cite{Took3}). Quaternions are the generalizations of real numbers and complex numbers.   Quaternion matrix decompositions have always been at the heart of color image processing and signal processing (\cite{N. LE Bihan}, \cite{quaternionapp01}).

In contrast to tensor decomposition and tensor equation over conventional algebra, the tensor decomposition and tensor equation over quaternion algebra are at present far from fully developed. Due to the noncommutativity, one can not directly extend various results on real numbers or complex numbers to quaternions. Hence, the research on tensor decomposition and tensor equation over quaternion algebra have not been more fruitful so far than those over fields. Motivated by the wide application of tensor decomposition, tensor equation and quaternion algebra and in order to improve the theoretical development
of the quaternion tensor theory, we in this paper consider some quaternion tensor decompositions and quaternion tensor equations.

The purpose of this paper is threefold. Firstly, we present the structures of quaternion tensor SVD, quaternion tensor rank decomposition, and decomposition for $\eta$-Hermitian quaternion tensor (see Definition \ref{30defi32}) with the isomorphic group structures and Einstein product. We investigate and discuss in detail the structures of these decompositions. Secondly, we define the Moore-Penrose inverse of even order quaternion tensors with the Einstein product. We give the expression of the Moore-Penrose inverse of an even order tensor by using the quaternion tensor SVD. Finally, we consider some quaternion tensor equations. We give some  necessary and sufficient conditions for the existence of a solution to the following generalized Sylvester quaternion tensor equation
 \begin{align}\label{tensorequ01}
\mathcal{A}*_{N}\mathcal{X}*_{M}\mathcal{B}+\mathcal{C}*_{N}\mathcal{Y}*_{M}\mathcal{D}=\mathcal{E},
\end{align}
in terms of the Moore-Penrose inverses of the coefficient quaternion tensors, where the operation $*_{N}$ is the Einstein product (see Definition \ref{03defi21}), and $\mathcal{A},~\mathcal{B},~\mathcal{C},~\mathcal{D}$ and $\mathcal{E}$ are given quaternion tensors with suitable order. We also present the expression of the general solution to the quaternion tensor equation (\ref{tensorequ01}) when it is solvable. As applications of (\ref{tensorequ01}), we provide some solvability conditions and general $\eta$-Hermitian solutions to some quaternion tensor equations.

\section{\textbf{Preliminaries}}

An order $N$ tensor $\mathcal{A}=(a_{i_{1}\cdots i_{N}})_{1\leq i_{j}\leq I_{j}}~(j=1,\ldots,N)$ is a multidimensional array with $I_{1}I_{2}\cdots I_{N}$ entries.
Let $\mathbb{R}^{I_{1}\times \cdots\times I_{N}},\mathbb{C}^{I_{1}\times \cdots\times I_{N}}$ and $\mathbb{H}^{I_{1}\times \cdots\times I_{N}}$ stand, respectively, for the sets of the order $N$ dimension $I_{1}\times \cdots\times I_{N}$ tensors over the real number field $\mathbb{R}$, the complex number field $\mathbb{C}$ and the real quaternion algebra
\[
\mathbb{H}%
=\big\{q_{0}+q_{1}\mathbf{i}+q_{2}\mathbf{j}+q_{3}\mathbf{k}\big|~\mathbf{i}^{2}=\mathbf{j}^{2}=\mathbf{k}^{2}=\mathbf{ijk}=-1,
q_{0},q_{1},q_{2},q_{3}\in\mathbb{R}\big\}.
\]
It is well known that the quaternion algebra is an associative and noncommutative division algebra. For more definitions and properties of quaternions, we refer the reader to the recent book \cite{rodman} and the survey paper \cite{F. Zhang}. The symbol $\bar{q}$ stands for the conjugate transpose of a quaternion $q=q_{0}+q_{1}\mathbf{i}+q_{2}\mathbf{j}+q_{3}\mathbf{k},q_{0},q_{1},q_{2},q_{3}\in\mathbb{R}$. Clearly, $\bar{q}=q_{0}-q_{1}\mathbf{i}-q_{2}\mathbf{j}-q_{3}\mathbf{k}.$

For a quaternion tensor $\mathcal{A}=(a_{i_{1}\cdots i_{N}j_{1}\cdots j_{M}})\in\mathbb{H}^{I_{1}\times \cdots \times I_{N}\times J_{1}\times \cdots \times J_{M}}$, let
$\mathcal{B}=(b_{i_{1}\cdots i_{M}j_{1}\cdots j_{N}})\in\mathbb{H}^{J_{1}\times \cdots \times J_{M}\times I_{1}\times \cdots \times I_{N}}$ be the conjugate transpose of $\mathcal{A}$, where $b_{i_{1}\cdots i_{M}j_{1}\cdots j_{N}}=\bar{a}_{i_{1}\cdots i_{N}j_{1}\cdots j_{M}}$. The tensor $\mathcal{B}$ is denoted by $\mathcal{A}^{*}$. A ``square'' tensor $\mathcal{A}\in\mathbb{H}^{I_{1}\times \cdots \times I_{N}\times I_{1}\times \cdots \times I_{N}}$ is said to be Hermitian if $\mathcal{A}=\mathcal{A}^{*}$. A ``square'' tensor $\mathcal{D}=(d_{i_{1}\cdots i_{N}i_{1}\cdots i_{N}})\in\mathbb{H}^{I_{1}\times \cdots \times I_{N}\times I_{1}\times \cdots \times I_{N}}$ is called a diagonal tensor if all its entries are zero except for $d_{i_{1}\cdots i_{N}i_{1}\cdots i_{N}}$. If all the diagonal entries $d_{i_{1}\cdots i_{N}i_{1}\cdots i_{N}}=1$, then $\mathcal{D}$ is a unit tensor, denoted by $\mathcal{I}$. The zero tensor with suitable order is denoted by $0$.

The definition of Einstein product is a contracted product which has been used in continuum mechanics \cite{lai01}. At first, we give the definition of Einstein product.

\begin{definition} [Einstein product] \cite{einstein}\label{03defi21}
For $\mathcal{A}\in\mathbb{H}^{I_{1}\times \cdots \times I_{N}\times J_{1}\times \cdots \times J_{N}}$ and $\mathcal{B}\in\mathbb{H}^{J_{1}\times \cdots \times J_{N}\times K_{1}\times \cdots \times K_{M}}$, the Einstein product of tensors $\mathcal{A}$ and $\mathcal{B}$ is defined by the operation $*_{N}$ via
\begin{align}
(\mathcal{A}*_{N}\mathcal{B})_{i_{1}\cdots i_{N}k_{1}\cdots k_{M}}=\sum_{j_{1}\cdots j_{N}}a_{i_{1}\cdots i_{N}j_{1}\cdots j_{N}}b_{j_{1}\cdots j_{N}k_{1}\cdots k_{M}},
\end{align}
where $\mathcal{A}*_{N}\mathcal{B}\in\mathbb{H}^{I_{1}\times \cdots \times I_{N}\times K_{1}\times \cdots \times K_{M}}$. The associative law of this tensor product holds.
\end{definition}

It is easy to infer the following results.
\begin{proposition}
Let $\mathcal{A}\in\mathbb{H}^{I_{1}\times \cdots \times I_{N}\times J_{1}\times \cdots \times J_{N}}$ and $\mathcal{B}\in\mathbb{H}^{J_{1}\times \cdots \times J_{N}\times K_{1}\times \cdots \times K_{M}}$. Then\\
$(1)$ $(\mathcal{A}*_{N}\mathcal{B})^{*}=\mathcal{B}^{*}*_{N}\mathcal{A}^{*}$;\\
$(2)$ $\mathcal{I}_{N}*_{N}\mathcal{B}=\mathcal{B}$ and $\mathcal{B}*_{M}\mathcal{I}_{M}=\mathcal{B}$, where unit tensors
$\mathcal{I}_{N}\in\mathbb{H}^{J_{1}\times \cdots \times J_{N}\times J_{1}\times \cdots \times J_{N}}$ and
$\mathcal{I}_{M}\in\mathbb{H}^{K_{1}\times \cdots \times K_{N}\times K_{1}\times \cdots \times K_{N}}$.
\end{proposition}

In 2013, M. Brazell, N. Li, C. Navasca, C. Tamon \cite{navasca} gave a transformation $f$ between the tensor and matrix over real number field. We can give the same definition of the transformation $f$ between the tensor and matrix over quaternion algebra.

\begin{definition} [transformation]\cite{navasca}\label{21deff}  Define the transformation $f:\mathbb{T}_{I_{1},I_{2},\cdots,I_{N},J_{1},J_{2},\cdots,J_{N}}(\mathbb{H})
\rightarrow\mathbb{M}_{I_{1}\cdot I_{2}\cdots I_{N-1}\cdot I_{N},J_{1}\cdot J_{2}\cdots J_{N-1}\cdot J_{N}}(\mathbb{H})$
 with $f(\mathcal{A})=A$ defined component-wise as
\begin{equation}
(\mathcal{A})_{i_{1}i_{2}\cdots i_{n}j_{1}j_{2}\cdots j_{n}}
\xrightarrow{f}(A)_{[i_{1}+\sum_{k=2}^{N}(i_{k}-1)\prod_{s=1}^{k-1}I_{s}]
[j_{1}+\sum_{k=2}^{N}(j_{k}-1)\prod_{s=1}^{k-1}J_{s}]},
\label{eq:transformation}
\end{equation}
where $\mathcal{A}\in \mathbb{T}_{I_{1},I_{2},\ldots,I_{N},J_{1},J_{2},\ldots,J_{N}}(\mathbb{H})$ and $A\in\mathbb{M}_{I_{1}\cdot I_{2} \cdots I_{N}, J_{1} \cdot J_{2} \cdots J_{N}}(\mathbb{H})$.
That is, the element is one-to-one correspondence between the quaternion matrix $A$ and the quaternion tensor $\mathcal{A}.$
\end{definition}

\begin{example}
Consider the $2\times 2\times 2\times 3$-quaternion tensor $\mathcal{A}$ defined by
\begin{align*}
a_{1111}=\mathbf{i},~a_{1112}=\mathbf{j},~a_{1113}=-\mathbf{i}+\mathbf{j},~a_{1121}=\mathbf{k},~a_{1122}=1+\mathbf{i},~a_{1123}=2\mathbf{j}+\mathbf{k},
\end{align*}
\begin{align*}
a_{1211}=\mathbf{i}+\mathbf{k},~a_{1212}=2\mathbf{i}+\mathbf{k},~a_{1213}=3\mathbf{k},~a_{1221}=\mathbf{j}-\mathbf{k},~a_{1222}=1+\mathbf{k},
~a_{1223}=3\mathbf{i}-2\mathbf{j},
\end{align*}
\begin{align*}
a_{2111}=\mathbf{k},~a_{2112}=3\mathbf{j},~a_{2113}=2-\mathbf{j},~a_{2121}=\mathbf{j}-\mathbf{k},~a_{2122}=1+\mathbf{j}+\mathbf{k},
~a_{2123}=\mathbf{j}+2\mathbf{k},
\end{align*}
\begin{align*}
a_{2211}=1+2\mathbf{i}+3\mathbf{j}-\mathbf{k},~a_{2212}=1+\mathbf{j}+\mathbf{k},~a_{2213}=2\mathbf{k},~a_{2221}=\mathbf{k},~a_{2222}=-\mathbf{i},
~a_{2223}=3+\mathbf{j}.
\end{align*}
Then
\begin{align*}
f(\mathcal{A})&=
\begin{pmatrix}
a_{1111}&a_{1121}&a_{1112}&a_{1122}&a_{1113}&a_{1123}\\
a_{2111}&a_{2121}&a_{2112}&a_{2122}&a_{2113}&a_{2123}\\
a_{1211}&a_{1221}&a_{1212}&a_{1222}&a_{1213}&a_{1223}\\
a_{2211}&a_{2221}&a_{2212}&a_{2222}&a_{2213}&a_{2223}
\end{pmatrix}
\\&=
\begin{pmatrix}
\mathbf{i}&\mathbf{k}&\mathbf{j}&1+\mathbf{i}&-\mathbf{i}+\mathbf{j}&2\mathbf{j}+\mathbf{k}\\
\mathbf{k}&\mathbf{j}-\mathbf{k}&3\mathbf{j}&1+\mathbf{j}+\mathbf{k}&2-\mathbf{j}&\mathbf{j}+2\mathbf{k}\\
\mathbf{i}+\mathbf{k}&\mathbf{j}-\mathbf{k}&2\mathbf{i}+\mathbf{k}&1+\mathbf{k}&3\mathbf{k}&3\mathbf{i}-2\mathbf{j}\\
1+2\mathbf{i}+3\mathbf{j}-\mathbf{k}&\mathbf{k}&1+\mathbf{j}+\mathbf{k}&-\mathbf{i}&2\mathbf{k}&3+\mathbf{j}
\end{pmatrix}.
\end{align*}
\end{example}

We now consider the properties of the map $f$ defined in Definition \ref{21deff}. M. Brazell, N. Li, C. Navasca, C. Tamon \cite{navasca} discussed several consequential results from the group structure and gave a proof of fourth order tensors over the real number field. By the similar method, we can give the following properties of the map $f$ over quaternion algebra.

\begin{lemma}\label{03lemma22}
 Let $\mathcal{A}\in\mathbb{H}^{I_{1}\times I_{2}\times\cdots I_{N}\times J_{1}\times J_{2}\times\cdots\times J_{N}}$,
$\mathcal{B}\in\mathbb{H}^{J_{1}\times J_{2}\times\cdots\times J_{N}\times L_{1}\times L_{2}\times\cdots\times L_{N}}$
and $f$ be the map defined in (\ref{eq:transformation}).
 Then the following properties hold:\\
\quad \quad \hspace{1cm} 1. The map $f$ is a bijection.
  Moreover, there exists a bijective inverse map $f^{-1}$:
  \begin{align}
  \mathbb{M}_{I_{1}\cdot I_{2}\cdots I_{N-1}\cdot I_{N},J_{1}\cdot J_{2}\cdots J_{N-1}\cdot J_{N}}(\mathbb{H})
  \xrightarrow{f^{-1}}\mathbb{T}_{I_{1},I_{2},\cdots,I_{N},J_{1},J_{2},\cdots,J_{N}}(\mathbb{H})
  \end{align}
2. The map satisfies $f(\mathcal{A}*_{N}\mathcal{B})=f(\mathcal{A})\cdot f(\mathcal{B})$,
  where $\cdot$ refers to the usual matrix multiplication.
\end{lemma}

\begin{proof}
Here we denote $\langle N\rangle=\{1,2,\cdots,N\}$ and its cardinality as $|N|$.\\
(1) By the definition $f$,
we can define three maps $h_{1},h_{2},$ and $h_{3}$, that is, $h_{1}:\langle I_{1}\rangle\times  \langle I_{2}\rangle\times\cdots\times \langle I_{N}\rangle
\rightarrow \langle I_{1}I_{2}\cdots I_{N}\rangle$ by
\begin{align}
h_{1}(i_{1},i_{2},\cdots,i_{N})=i_{1}+\sum_{k=2}^{N}(i_{k}-1)\prod_{s=1}^{k-1}I_{s},
\end{align} $h_{2}:\langle J_{1}\rangle\times \langle J_{2}\rangle\times\cdots\times \langle J_{N}\rangle
\rightarrow \langle J_{1}J_{2}\cdots J_{N}\rangle$ by
\begin{align}
h_{2}(j_{1},j_{2},\cdots,j_{N})=j_{1}+\sum_{k=2}^{N}(j_{k}-1)\prod_{s=1}^{k-1}J_{s},
\end{align}
$h_{3}: \langle L_{1}\rangle\times \langle L_{2}\rangle\times\cdots\times \langle L_{N}\rangle
\rightarrow \langle L_{1}L_{2}\cdots L_{N}\rangle$ by
\begin{align}
h_{3}(l_{1},l_{2},\cdots,l_{N})=l_{1}+\sum_{k=2}^{N}(l_{k}-1)\prod_{s=1}^{k-1}L_{s}.
\end{align}
Clearly, the maps $h_{1},h_{2},$ and $h_{3}$ are bijections since $f$ is a bijection.\\
(2) Since $f$ is a bijection,
for some $1\leq i\leq I_{1}I_{2}\cdots I_{N}, 1\leq l\leq L_{1}L_{2}\cdots L_{N},$
there exist some unique indices $i_{1},i_{2},\cdots,i_{N},l_{1},l_{2},\cdots,l_{N}$
for $1\leq i_{1}\leq I_{1}, 1\leq i_{2}\leq I_{2},\cdots, 1\leq i_{N}\leq I_{N},
1\leq l_{1}\leq L_{1}, 1\leq l_{2}\leq L_{2},\cdots, 1\leq l_{N}\leq L_{N}$
such that
\begin{align}
i_{1}+\sum_{k=2}^{N}(i_{k}-1)\prod_{s=1}^{k-1}I_{s}=i
\end{align}
and
\begin{align}
l_{1}+\sum_{k=2}^{N}(l_{k}-1)\prod_{s=1}^{k-1}L_{s}=l.
\end{align} Hence, we have that
\begin{equation*}
[f(\mathcal{A}*_{N}\mathcal{B})]_{il}
=(\mathcal{A}*_{N}\mathcal{B})_{i_{1}i_{2}\cdots i_{N}l_{1}l_{2}\cdots l_{N}}
=\sum_{j_{1}j_{2}\cdots j_{N}}a_{i_{1}i_{2}\cdots i_{N}j_{1}j_{2}\cdots j_{N}}
b_{j_{1}j_{2}\cdots j_{N}l_{1}l_{2}\cdots l_{N}},
\end{equation*}
\begin{equation*}
[f(\mathcal{A})\cdot f(\mathcal{B})]_{il}=\sum_{j=1}^{|J_{1}J_{2}\cdots J_{N}|}[f(\mathcal{A})]_{ij}[f(\mathcal{B})]_{jl}.
\end{equation*}
For every $1\leq j\leq J_{1}J_{2}\cdots J_{N}$, there exist some unique $j_{1},j_{2},\cdots,j_{N}$
such that
\begin{align}
j_{1}+\sum_{k=2}^{N}(j_{k}-1)\prod_{s=1}^{k-1}J_{s}=j.
\end{align} Then, we have that
\begin{equation*}
\sum_{j_{1}j_{2}\cdots j_{N}}a_{i_{1}i_{2}\cdots i_{N}j_{1}j_{2}\cdots j_{N}}b_{j_{1}j_{2}\cdots j_{N}l_{1}l_{2}\cdots l_{N}}
=\sum_{j=1}^{|J_{1}J_{2}\cdots J_{N}|}[f(\mathcal{A})]_{ij}[f(\mathcal{B})]_{jl}.
\end{equation*}Hence, $f(\mathcal{A}*_{N}\mathcal{B})=f(\mathcal{A})\cdot f(\mathcal{B})$.

\end{proof}

It follows from Lemma \ref{03lemma22} that the Einstein product can be defined through the transformation:
\begin{equation}
\mathcal{A}*_{N}\mathcal{B}=f^{-1}[f(\mathcal{A}*_{N}\mathcal{B})]=f^{-1}[f(\mathcal{A})\cdot f(\mathcal{B})].
\end{equation}
Consequently, the inverse map $f^{-1}$ satisfies
\begin{equation}\label{21equ211}
f^{-1}(A\cdot B)=f^{-1}(A)*_{N}f^{-1}(B),
\end{equation}where $A$ and $B$ are quaternion matrices with appropriate sizes.

The following lemma appeared in \cite{navasca} over the real number field. Using similar methods, we can extend it to the quaternion algebra.
\begin{lemma}\label{21lemma23}
\cite{navasca} Suppose $(\mathbb{M},\cdot)$ is a group. Let $f:\mathbb{T}\rightarrow \mathbb{M}$
be any bijection. Then we can define a group structure on $\mathbb{T}$ by defining
\begin{equation*}
\mathcal{A}*_{N}\mathcal{B}=f^{-1}[f(\mathcal{A})\cdot f(\mathcal{B})]
\end{equation*}
for all $\mathcal{A},\mathcal{B}\in\mathbb{T}$. Moreover, the mapping $f$ is an isomorphism.
\end{lemma}

\section{\textbf{Decompositions for quaternion tensors via isomorphic group structures}}

In this section, we investigate some decompositions for quaternion tensors via isomorphic group structures. This section is organized as follows:

\begin{itemize}
 \item \S 3.1.  Give the SVD for a quaternion tensor.
 \item \S 3.2.   Give the rank decompositions for quaternion tensors.
 \item \S 3.2.  Give a decomposition for an $\eta$-Hermitian quaternion tensor.
\end{itemize}

\subsection{\textbf{SVD for quaternion tensors}}

In this section, we consider the singular value decomposition for quaternion tensors via isomorphic group structures. The SVD for quaternion matrices was given in \cite{F. Zhang}.
\begin{lemma} \cite{F. Zhang} (Quaternion matrix SVD)\label{15lemma31}
Let $A\in \mathbb{H}^{m\times n}$ be of rank $r$. Then there exist unitary quaternion matrices $U\in \mathbb{H}^{m\times m}$ and $V\in \mathbb{H}^{n\times n}$ such that
\begin{align}
A=U\begin{pmatrix}D_{r}&0\\0&0\end{pmatrix}V^{*},
\end{align}where $D_{r}=diag(d_{1},\ldots,d_{r})$ and the $d_{i} ~ (i=1,\ldots,r)$ are real positive singular values of $A$.
\end{lemma}

Now we give the definition of the unitary quaternion tensor.

\begin{definition}
[unitary quaternion tensor] A tensor $\mathcal{U}\in\mathbb{H}^{I_{1}\times \cdots \times I_{N}\times I_{1}\times \cdots \times I_{N}}$ is unitary if
$\mathcal{U}*_{N}\mathcal{U}^{*}=\mathcal{U}^{*}*_{N}\mathcal{U}=\mathcal{I}.$
\end{definition}

We give quaternion tensor SVD.
\begin{theorem}(Quaternion tensor SVD)\label{23theorem32}
Let $\mathcal{A}\in\mathbb{H}^{I_{1}\times \cdots \times I_{N}\times J_{1}\times \cdots \times J_{N}}$ with $r=\mbox{rank} (f(\mathcal{A}))$, where $f$ is the transformation in (\ref{eq:transformation}).  Then there exist
 unitary tensors $\mathcal{U}\in\mathbb{H}^{I_{1}\times \cdots \times I_{N}\times I_{1}\times \cdots \times I_{N}}$ and $\mathcal{V}\in\mathbb{H}^{J_{1}\times \cdots \times J_{N}\times J_{1}\times \cdots \times J_{N}}$ such that
\begin{align}
\mathcal{A}=\mathcal{U}*_{N}\mathcal{B}*_{N}\mathcal{V}^{*},
\end{align}where the tensor $\mathcal{B}\in\mathbb{R}^{I_{1}\times \cdots \times I_{N}\times J_{1}\times \cdots \times J_{N}}$ has the following structure
\begin{align}\label{23equ033}
\mathcal{B}_{i_{1}\cdots  i_{N}j_{1}\cdots j_{N}}=
 \left\{\begin{array}{c}
 d_{i}, ~\mbox{if}~(i_{1}\cdots  i_{N}j_{1}\cdots j_{N})=(p_{1}^{i}\cdots  p_{N}^{i}q_{1}^{i}\cdots q_{N}^{i}),\\
 0, \mbox{otherwise},
 \end{array}
  \right.
\end{align}
$d_{i}~ (i=1,\ldots,r)$ are the positive singular values of the quaternion matrix $f(\mathcal{A})$ and
\begin{align}\label{23equ034}
p_{N}^{i}=\left[\frac{i-1}{\prod_{s=1}^{N-1}I_{s}}\right]+1,
~
p_{N-1}^{i}=\left[\frac{i-1-(p_{N}^{i}-1)\prod_{s=1}^{N-1}I_{s}}{\prod_{s=1}^{N-2}I_{s}}\right]+1,
\end{align}
\begin{align}
\vdots
\end{align}
\begin{align}
p_{t}^{i}=\left[\frac{i-1-\sum_{k=t+1}^{N}(p_{k}^{i}-1)\prod_{s=1}^{k-1}I_{s}}{\prod_{s=1}^{t-1}I_{s}}\right]+1,
\end{align}
\begin{align}
\vdots
\end{align}
\begin{align}
p_{2}^{i}=\left[\frac{i-1-\sum_{k=3}^{N}(p_{k}^{i}-1)\prod_{s=1}^{k-1}I_{s}}{I_{1}}\right]+1,
~
p_{1}^{i}=i-\sum_{k=2}^{N}(p_{k}^{i}-1)\prod_{s=1}^{k-1}I_{s},
\end{align}
\begin{align}
q_{N}^{i}=\left[\frac{i-1}{\prod_{s=1}^{N-1}J_{s}}\right]+1,
~
q_{N-1}^{i}=\left[\frac{i-1-(q_{N}^{i}-1)\prod_{s=1}^{N-1}J_{s}}{\prod_{s=1}^{N-2}J_{s}}\right]+1,
\end{align}
\begin{align}
\vdots
\end{align}
\begin{align}
q_{t}^{i}=\left[\frac{i-1-\sum_{k=t+1}^{N}(q_{k}^{i}-1)\prod_{s=1}^{k-1}J_{s}}{\prod_{s=1}^{t-1}J_{s}}\right]+1,
\end{align}
\begin{align}
\vdots
\end{align}
\begin{align}\label{23equ313}
q_{2}^{i}=\left[\frac{i-1-\sum_{k=3}^{N}(q_{k}^{i}-1)\prod_{s=1}^{k-1}J_{s}}{J_{1}}\right]+1,
~
q_{1}^{i}=i-\sum_{k=2}^{N}(q_{k}^{i}-1)\prod_{s=1}^{k-1}J_{s},
\end{align}
and $\left[ a \right]$ is the largest integer less than or equal to the real number $a$.
\end{theorem}

\begin{proof}
Let $A=f(\mathcal{A})$. It follows from Lemma \ref{15lemma31} that
\begin{align}
A=U\begin{pmatrix}D_{r}&0\\0&0\end{pmatrix}V^{*},
\end{align}where $U$ and $V$ are unitary quaternion matrices, $D_{r}=\mbox{diag} (d_{1},\ldots,d_{r})$ and $d_{i}~ (i=1,\ldots,r)$ are the positive singular values of $A$. From the property of $f$ in (\ref{21equ211}) and Lemma \ref{21lemma23}, we have
\begin{align}
f^{-1}(A)=f^{-1}\left(U\begin{pmatrix}D_{r}&0\\0&0\end{pmatrix}V^{*} \right)=f^{-1}(U)*_{N}\left(f^{-1}\begin{pmatrix}D_{r}&0\\0&0\end{pmatrix}\right)*_{N}f^{-1}(V^{*})
\end{align}
\begin{align}
\Longrightarrow
\mathcal{A}=\mathcal{U}*_{N}\mathcal{B}*_{N}\mathcal{V}^{*},
\end{align}
where $\mathcal{U}=f^{-1}(U)$, $\mathcal{V}^{*}=f^{-1}(V^{*})$ and $\mathcal{B}=f^{-1}\begin{pmatrix}D_{r}&0\\0&0\end{pmatrix}$ is a real tensor whose nonzero entries are $d_{i}$. Note that
\begin{align}
\mathcal{U}*_{N}\mathcal{U}^{*}=f^{-1}(U)*_{N}f^{-1}(U)^{*}=f^{-1}(U)*_{N}f^{-1}(U^{*})=f^{-1}(UU^{*})=f^{-1}(I)=\mathcal{I},
\end{align}
\begin{align}
\mathcal{U}^{*}*_{N}\mathcal{U}=f^{-1}(U)^{*}*_{N}f^{-1}(U)=f^{-1}(U^{*})*_{N}f^{-1}(U)=f^{-1}(U^{*}U)=f^{-1}(I)=\mathcal{I}.
\end{align}
Similarly, we can prove that
\begin{align}
\mathcal{V}*_{N}\mathcal{V}^{*}=\mathcal{V}^{*}*_{N}\mathcal{V}=\mathcal{I}.
\end{align}
Hence, $\mathcal{U}$ and $\mathcal{V}$ are unitary quaternion tensors.

Now we consider the structure of the tensor $\mathcal{B}\in\mathbb{R}^{I_{1}\times \cdots \times I_{N}\times J_{1}\times \cdots \times J_{N}}$. Note that
\begin{align}
\mathcal{B}=f^{-1}\begin{pmatrix}D_{r}&0\\0&0\end{pmatrix}=f^{-1}
\bordermatrix{
~& & & & &\Pi_{s=1}^{N}J_{s}-r \cr
~&d_{1}&0&\cdots&0&0\cr
~&0&d_{2}&\cdots&0&0\cr
~&\vdots&\vdots&\ddots&\vdots&\vdots\cr
~&0&0&\cdots&d_{r}&0\cr
\Pi_{s=1}^{N}I_{s}-r&0&0&\cdots&0&0
}.
\end{align} The subscript of $d_{i}$ in the diagonal matrix $\begin{pmatrix}D_{r}&0\\0&0\end{pmatrix}$ is $ii$. We want to find the subscript of $d_{i}$ in the tensor $\mathcal{B}.$ Assume that the subscript of $d_{i}$ in the tensor $\mathcal{B}$ is
$(p_{1}^{i}\cdots  p_{N}^{i}q_{1}^{i}\cdots q_{N}^{i})$. From the Definition (\ref{21deff}), we obtain that
\begin{align}\label{21equ321}
p_{1}^{i}+\sum_{k=2}^{N}(p_{k}^{i}-1)\prod_{s=1}^{k-1}I_{s}=i,
\end{align}
and
\begin{align}
q_{1}^{i}+\sum_{k=2}^{N}(q_{k}^{i}-1)\prod_{s=1}^{k-1}J_{s}=i.
\end{align}It follows from equation (\ref{21equ321}) that
\begin{align}
(p_{N}^{i}-1)\prod_{s=1}^{N-1}I_{s}=&i-p_{1}^{i}-(p_{2}^{i}-1)I_{1}-(p_{3}^{i}-1)I_{1}I_{2}-\cdots-(p_{N-1}^{i}-1)\prod_{s=1}^{N-2}I_{s}
\nonumber\\ \leq &i-1,
\end{align}
i.e.,
\begin{align}
p_{N}^{i}\leq \frac{i-1}{\prod_{s=1}^{N-1}I_{s}}+1.
\end{align}
Since $p_{N}^{i}$ is a positive integer, we see that
\begin{align}
p_{N}^{i}\leq \left[\frac{i-1}{\prod_{s=1}^{N-1}I_{s}}\right]+1.
\end{align}
We now want to prove that $p_{N}^{i}=\left[\frac{i-1}{\prod_{s=1}^{N-1}I_{s}}\right]+1.$ Suppose $p_{N}^{i}\leq\left[\frac{i-1}{\prod_{s=1}^{N-1}I_{s}}\right].$ Then we have
\begin{align}
i&=p_{1}^{i}+\sum_{k=2}^{N}(p_{k}^{i}-1)\prod_{s=1}^{k-1}I_{s}\nonumber\\
&\leq I_{1}+(I_{2}-1)I_{1}+(I_{3}-1)I_{1}I_{2}+\cdots+(I_{N-1}-1)\prod_{s=1}^{N-2}I_{s}+
\left(\left[\frac{i-1}{\prod_{s=1}^{N-1}I_{s}}\right]-1\right)\prod_{s=1}^{N-1}I_{s}\nonumber
\\&= \left[\frac{i-1}{\prod_{s=1}^{N-1}I_{s}}\right] \prod_{s=1}^{N-1}I_{s}\leq i-1.
\end{align}This is a contradiction. Hence, $p_{N}^{i}=\left[\frac{i-1}{\prod_{s=1}^{N-1}I_{s}}\right]+1.$ Substituting $p_{N}^{i}$ into (\ref{21equ321}) yields
\begin{align}
p_{1}^{i}+\sum_{k=2}^{N-1}(p_{k}^{i}-1)\prod_{s=1}^{k-1}I_{s}=i-(p_{N}^{i}-1)\prod_{s=1}^{N-1}I_{s}.
\end{align}
Hence, we can obtain the expression of $p_{N-1}^{i}$ by using the same method. Similarly, we can give the expressions of $p_{1}^{i},\ldots,p_{N-2}^{i},$ and $q_{1}^{i},\ldots,q_{N}^{i}.$

\end{proof}

\begin{remark}
 Theorem \ref{23theorem32} contains the existing generalization of the SVD for a $2\times 2\times 2\times 2$-real tensor $\mathcal{A}$ \cite{navasca} as a special case. Moreover, we present the positions of the positive singular values of the quaternion matrix $f(\mathcal{A})$ in the real tensor $\mathcal{B}$.
\end{remark}

\begin{example}\label{27example2}
Consider the SVD for the $2\times 2\times 3\times 2$-quaternion tensor $\mathcal{A}$ defined by
\begin{align*}
a_{1111}=0,~a_{1121}=0,~a_{1131}=\frac{\sqrt{2}}{2}\mathbf{j},~a_{1112}=\frac{\sqrt{2}}{2}\mathbf{j},~a_{1122}=0,~a_{1132}=0,
\end{align*}
\begin{align*}
a_{2111}=-\mathbf{i},~a_{2121}=-\mathbf{i},~a_{2131}=0,~a_{2112}=0,~a_{2122}=0,~a_{2132}=0,
\end{align*}
\begin{align*}
a_{1211}=\mathbf{j},~a_{1221}=-\mathbf{j},~a_{1231}=-\mathbf{j},~a_{1212}=\mathbf{j},~a_{1222}=0,~a_{1232}=0,
\end{align*}
\begin{align*}
a_{2211}=\frac{\sqrt{2}}{2}(-1+\mathbf{i}+\mathbf{j}-\mathbf{k}),~a_{2221}=\frac{\sqrt{2}}{2}(1-\mathbf{i}-\mathbf{j}+\mathbf{k}),
~a_{2231}=\frac{\sqrt{2}}{2}(-1+\mathbf{i}+\mathbf{j}-\mathbf{k}),
\end{align*}
\begin{align*}
a_{2212}=\frac{\sqrt{2}}{2}(1-\mathbf{i}-\mathbf{j}+\mathbf{k}),
~a_{2222}=-1-\mathbf{i}+\mathbf{j}-\mathbf{k},~a_{2232}=-1-\mathbf{i}+\mathbf{j}-\mathbf{k}.
\end{align*}
Then we have
\begin{align}
\mathcal{A}=\mathcal{U}*_{2}\mathcal{B}*_{2}\mathcal{V}^{*},
\end{align}where $\mathcal{U}\in \mathbb{H}^{2\times 2\times 2\times 2}$ and $\mathcal{V}\in \mathbb{H}^{3\times 2\times 3\times 2}$ are unitary quaternion tensors, $\mathcal{B}\in \mathbb{R}^{2\times 2\times 3\times 2}$ is a real tensor and
\begin{align*}
\mathcal{B}_{i_{1}i_{2}j_{1}j_{2}}=
 \left\{\begin{array}{c}
 1, ~\mbox{if}~(i_{1}i_{2}j_{1}j_{2})=(1111),\\
 \sqrt{2}, ~\mbox{if}~(i_{1}i_{2}j_{1}j_{2})=(2121),\\
 2, ~\mbox{if}~(i_{1}i_{2}j_{1}j_{2})=(1231),\\
 4, ~\mbox{if}~(i_{1}i_{2}j_{1}j_{2})=(2212),\\
 0,~\mbox{otherwise},
 \end{array}
  \right.
\end{align*}
\begin{align*}
\mathcal{U}_{1111}=\mathbf{i},~\mathcal{U}_{1121}=0,
~\mathcal{U}_{1112}=0,~\mathcal{U}_{1122}=0,
\end{align*}
\begin{align*}
\mathcal{U}_{2111}=0,~\mathcal{U}_{2121}=\mathbf{j},
~\mathcal{U}_{2112}=0,~\mathcal{U}_{2122}=0,
\end{align*}
\begin{align*}
\mathcal{U}_{1211}=0,~\mathcal{U}_{1221}=0,
~\mathcal{U}_{1212}=\mathbf{k},~\mathcal{U}_{1222}=0,
\end{align*}
\begin{align*}
\mathcal{U}_{2211}=0,~\mathcal{U}_{2221}=0,
~
\mathcal{U}_{2212}=0,
~\mathcal{U}_{2222}=\frac{\sqrt{2}}{2}+\frac{\sqrt{2}}{2}\mathbf{j},
\end{align*}
\begin{align*}
\mathcal{V}^{*}_{1111}=0,~\mathcal{V}^{*}_{1121}=0,
~\mathcal{V}^{*}_{1131}=-\frac{\sqrt{2}}{2}\mathbf{k},~\mathcal{V}^{*}_{1112}=-\frac{\sqrt{2}}{2}\mathbf{k},
~\mathcal{V}^{*}_{1122}=0,~\mathcal{V}^{*}_{1132}=0,
\end{align*}
\begin{align*}
\mathcal{V}^{*}_{2111}=-\frac{\sqrt{2}}{2}\mathbf{k},,~\mathcal{V}^{*}_{2121}=-\frac{\sqrt{2}}{2}\mathbf{k},,
~\mathcal{V}^{*}_{2131}=0,~\mathcal{V}^{*}_{2112}=0,
~\mathcal{V}^{*}_{2122}=0,~\mathcal{V}^{*}_{2132}=0,
\end{align*}
\begin{align*}
\mathcal{V}^{*}_{3111}=\frac{1}{2}\mathbf{i},~\mathcal{V}^{*}_{3121}=-\frac{1}{2}\mathbf{i},
~\mathcal{V}^{*}_{3131}=-\frac{1}{2}\mathbf{i},~\mathcal{V}^{*}_{3112}=\frac{1}{2}\mathbf{i},
~\mathcal{V}^{*}_{3122}=0,~\mathcal{V}^{*}_{3132}=0,
\end{align*}
\begin{align*}
\mathcal{V}^{*}_{1211}=\frac{1}{4}(\mathbf{i}+\mathbf{j}),~\mathcal{V}^{*}_{1221}=-\frac{1}{4}(\mathbf{i}+\mathbf{j}),
~\mathcal{V}^{*}_{1231}=\frac{1}{4}(\mathbf{i}+\mathbf{j}),~\mathcal{V}^{*}_{1212}=-\frac{1}{4}(\mathbf{i}+\mathbf{j}),
\end{align*}
\begin{align*}
~\mathcal{V}^{*}_{1222}=\frac{\sqrt{2}}{4}(\mathbf{j}-\mathbf{k}),~\mathcal{V}^{*}_{1232}=\frac{\sqrt{2}}{4}(\mathbf{j}-\mathbf{k}),
\mathcal{V}^{*}_{2211}=\frac{1}{4}(1-\mathbf{k}),~\mathcal{V}^{*}_{2221}=\frac{1}{4}(-1+\mathbf{k}),
\end{align*}
\begin{align*}
~\mathcal{V}^{*}_{2231}=\frac{1}{4}(1-\mathbf{k}),~\mathcal{V}^{*}_{2212}=\frac{1}{4}(-1+\mathbf{k}),
~\mathcal{V}^{*}_{2222}=\frac{\sqrt{2}}{4}(\mathbf{j}+\mathbf{k}),~\mathcal{V}^{*}_{2232}=\frac{\sqrt{2}}{4}(\mathbf{j}+\mathbf{k}),
\end{align*}
\begin{align*}
\mathcal{V}^{*}_{3211}=0,~\mathcal{V}^{*}_{3221}=0,
~\mathcal{V}^{*}_{3231}=0,~\mathcal{V}^{*}_{3212}=0,
~\mathcal{V}^{*}_{3222}=\frac{\sqrt{2}}{4}+\frac{\sqrt{6}}{4}\mathbf{i},~\mathcal{V}^{*}_{3232}=-\frac{\sqrt{2}}{4}+\frac{\sqrt{6}}{4}\mathbf{i}.
\end{align*}
\end{example}

\begin{remark}
Example \ref{27example2} shows that the real tensor $\mathcal{B}$ does not necessarily satisfy the following condition
 \begin{align*}
\mathcal{B}_{i_{1}\cdots  i_{N}j_{1}\cdots j_{N}}=
 0~~\mbox{if}~~(i_{1}\cdots  i_{N})\neq(j_{1}\cdots j_{N}).
\end{align*}
\end{remark}

\subsection{\textbf{Rank decompositions for quaternion tensors }}
In this section, we consider the rank decompositions for quaternion tensors.  The rank decomposition for a quaternion matrix is given in the following lemma.

\begin{lemma} \cite{rodman} (Quaternion matrix rank decomposition)\label{15lemma33}
Let $A\in \mathbb{H}^{m\times n}$ be of rank $r$. Then there exist invertible quaternion matrices $P\in \mathbb{H}^{m\times m}$ and $Q\in \mathbb{H}^{n\times n}$ such that
\begin{align}
A=P\begin{pmatrix}I_{r}&0\\0&0\end{pmatrix}Q.
\end{align}
\end{lemma}

We give the definition of tensor inverse.
\begin{definition}
[inverse of an even order tensor] A tensor $\mathcal{X}\in\mathbb{H}^{I_{1}\times \cdots \times I_{N}\times I_{1}\times \cdots \times I_{N}}$ is called the inverse of
$\mathcal{A}\in\mathbb{H}^{I_{1}\times \cdots \times I_{N}\times I_{1}\times \cdots \times I_{N}}$ if it satisfies $\mathcal{A}*_{N}\mathcal{X}=\mathcal{X}*_{N}\mathcal{A}=\mathcal{I}.$ It is denoted by $\mathcal{A}^{-1}.$
\end{definition}

Now we present the quaternion tensor rank decomposition.

\begin{theorem}(Quaternion tensor rank decomposition)
Let $\mathcal{A}\in\mathbb{H}^{I_{1}\times \cdots \times I_{N}\times J_{1}\times \cdots \times J_{N}}$ with $r=\mbox{rank} (f(\mathcal{A}))$, where $f$ is the transformation in (\ref{eq:transformation}).  Then there exist invertible quaternion tensors $\mathcal{P}\in\mathbb{H}^{I_{1}\times \cdots \times I_{N}\times I_{1}\times \cdots \times I_{N}}$ and $\mathcal{Q}\in\mathbb{H}^{J_{1}\times \cdots \times J_{N}\times J_{1}\times \cdots \times J_{N}}$ such that
\begin{align}
\mathcal{A}=\mathcal{P}*_{N}\mathcal{B}*_{N}\mathcal{Q}
\end{align}where $\mathcal{B}\in\mathbb{R}^{I_{1}\times \cdots \times I_{N}\times J_{1}\times \cdots \times J_{N}}$ has the following structure
\begin{align}
\mathcal{B}_{i_{1}\cdots  i_{N}j_{1}\cdots j_{N}}=
 \left\{\begin{array}{c}
 1, ~\mbox{if}~(i_{1}\cdots  i_{N}j_{1}\cdots j_{N})=(p_{1}^{i}\cdots  p_{N}^{i}q_{1}^{i}\cdots q_{N}^{i}),\\
 0, \mbox{otherwise},
 \end{array}
  \right.
\end{align}
and
\begin{align}
i\in\{1,\ldots,r\},~p_{N}^{i}=\left[\frac{i-1}{\prod_{s=1}^{N-1}I_{s}}\right]+1,
~
p_{N-1}^{i}=\left[\frac{i-1-(p_{N}^{i}-1)\prod_{s=1}^{N-1}I_{s}}{\prod_{s=1}^{N-2}I_{s}}\right]+1,
\end{align}
\begin{align}
\vdots
\end{align}
\begin{align}
p_{t}^{i}=\left[\frac{i-1-\sum_{k=t+1}^{N}(p_{k}^{i}-1)\prod_{s=1}^{k-1}I_{s}}{\prod_{s=1}^{t-1}I_{s}}\right]+1,
\end{align}
\begin{align}
\vdots
\end{align}
\begin{align}
p_{2}^{i}=\left[\frac{i-1-\sum_{k=3}^{N}(p_{k}^{i}-1)\prod_{s=1}^{k-1}I_{s}}{I_{1}}\right]+1,
~
p_{1}^{i}=i-\sum_{k=2}^{N}(p_{k}^{i}-1)\prod_{s=1}^{k-1}I_{s}.
\end{align}

\end{theorem}

\begin{proof}
Let $A=f(\mathcal{A})$ and $r=\mbox{rank} (f(\mathcal{A}))$. It follows from Lemma \ref{15lemma33} that
\begin{align}
A=P\begin{pmatrix}I_{r}&0\\0&0\end{pmatrix}Q,
\end{align}where $P$ and $Q$ are invertible quaternion matrices. From the property of $f$ in (\ref{21equ211}) and Lemma \ref{21lemma23}, we see that
\begin{align}
f^{-1}(A)=f^{-1}\left(P\begin{pmatrix}I_{r}&0\\0&0\end{pmatrix}Q\right)=f^{-1}(P)*_{N}\left(f^{-1}\begin{pmatrix}I_{r}&0\\0&0\end{pmatrix}\right)*_{N}f^{-1}(Q)
\end{align}
\begin{align}
\Longrightarrow
\mathcal{A}=\mathcal{P}*_{N}\mathcal{B}*_{N}\mathcal{Q},
\end{align}
where $\mathcal{P}=f^{-1}(P)$, $\mathcal{Q}=f^{-1}(Q)$ and $\mathcal{B}=f^{-1}\begin{pmatrix}I_{r}&0\\0&0\end{pmatrix}$ is a real tensor whose nonzero entries are 1. Note that
\begin{align}
\mathcal{I}=f^{-1}(I)=f^{-1}(PP^{-1})=f^{-1}(P)*_{N}f^{-1}(P^{-1})=\mathcal{P}*_{N}f^{-1}(P^{-1}),
\end{align}
\begin{align}
\mathcal{I}=f^{-1}(I)=f^{-1}(P^{-1}P)=f^{-1}(P^{-1})*_{N}f^{-1}(P)=f^{-1}(P^{-1})*_{N}\mathcal{P}.
\end{align}Hence, $\mathcal{P}$ is an invertible quaternion tensor.
Similarly, we can prove that $\mathcal{Q}$ is also an invertible quaternion tensor. We can give the subscript of $1$ in the tensor $\mathcal{B}$ by using the method that in Theorem \ref{23theorem32}.
\end{proof}

\subsection{\textbf{Decomposition for an $\eta$-Hermitian quaternion tensor}}
In this section, we discuss the decomposition for an $\eta$-Hermitian quaternion tensor. Let us turn back to the definition of $\eta$-Hermitian matrix. $\eta$-Hermitian matrix was first proposed in \cite{Took4}, and further discussed in \cite{hewangamc2017}-\cite{FZhang3}.
\begin{definition}[$\eta$-Hermitian matrix] \cite{Took4}
For $\eta\in\{\mathbf{i},\mathbf{j},\mathbf{k}\}$, a square real quaternion matrix $A$ is said to be $\eta$-Hermitian if $A=A^{\eta*},$ where $A^{\eta*}=-\eta A^{*}\eta$.
\end{definition}

The $\eta$-Hermitian matrices can be used in statistical signal processing and widely linear modelling  (\cite{Took1}-\cite{Took4}). Horn and Zhang \cite{FZhang3} presented an analogous special singular value decomposition for an $\eta$-Hermitian matrix.
\begin{lemma}
\label{14lemma35}
\cite{FZhang3}
Suppose that $A$ is $\eta$-Hermitian. Then there is a unitary matrix $U$ and a real nonnegative diagonal matrix $\Sigma$ such that
\begin{align*}
A=U\Sigma U^{\eta*}.
\end{align*}The diagonal entries of $\Sigma$ are the singular values of $A$.
\end{lemma}

Motivated by the wide application of $\eta$-Hermitian matrices, we define the $\eta$-Hermitian tensor as follows.

\begin{definition}[$\eta$-Hermitian tensor]\label{30defi32}
For $\eta\in \{\mathbf{i},\mathbf{j},\mathbf{k}\}$, a square quaternion tensor $\mathcal{A}\in\mathbb{H}^{I_{1}\times \cdots \times I_{N}\times I_{1}\times \cdots \times I_{N}}$ is said to be $\eta$-Hermitian if $\mathcal{A}=\mathcal{A}^{\eta*}$, where $\mathcal{A}^{\eta*}=-\eta \mathcal{A}^{*} \eta.$
\end{definition}

\begin{example}
Consider the $2\times 2\times 2\times 2$-quaternion tensor $\mathcal{A}$ defined by
\begin{align*}
a_{1111}=\mathbf{j}+\mathbf{k},~a_{1112}=1+\mathbf{i}+\mathbf{j},~a_{1121}=2\mathbf{i}+\mathbf{k},~a_{1122}=\mathbf{i}+\mathbf{j}+\mathbf{k},
\end{align*}
\begin{align*}
a_{1211}=1-\mathbf{i}+\mathbf{j},~a_{1212}=1+\mathbf{j},~a_{1221}=\mathbf{j}+\mathbf{k},~a_{1222}=\mathbf{i},
\end{align*}
\begin{align*}
a_{2111}=-2\mathbf{i}+\mathbf{k},~a_{2112}=\mathbf{j}+\mathbf{k},~a_{2121}=2\mathbf{k},~a_{2122}=2\mathbf{i}-\mathbf{j},
\end{align*}
\begin{align*}
a_{2211}=-\mathbf{i}+\mathbf{j}+\mathbf{k},~a_{2212}=-\mathbf{i},~a_{2221}=-2\mathbf{i}-\mathbf{j},~a_{2222}=1+\mathbf{j}-\mathbf{k}.
\end{align*}
Direct computation yields $\mathcal{A}=\mathcal{A}^{\mathbf{i}*}=-\mathbf{i}\mathcal{A}^{*}\mathbf{i}$. Hence, the quaternion tensor $\mathcal{A}$ is $\mathbf{i}$-Hermitian.
\end{example}

\begin{proposition}
Let $\mathcal{A}\in\mathbb{H}^{I_{1}\times \cdots \times I_{N}\times J_{1}\times \cdots \times J_{N}}$ and $\mathcal{B}\in\mathbb{H}^{J_{1}\times \cdots \times J_{N}\times K_{1}\times \cdots \times K_{M}}$. Then
$$(\mathcal{A}*_{N}\mathcal{B})^{\eta*}=\mathcal{B}^{\eta*}*_{N}\mathcal{A}^{\eta*}.$$
\end{proposition}

Now we give the decomposition for an $\eta$-Hermitian quaternion tensor.
\begin{theorem}(Decomposition for an $\eta$-Hermitian quaternion tensor)
Let $\mathcal{A}=\mathcal{A}^{\eta*}\in\mathbb{H}^{I_{1}\times \cdots \times I_{N}\times I_{1}\times \cdots \times I_{N}}$ with $r=\mbox{rank} (f(\mathcal{A}))$, where $f$ is the transformation in (\ref{eq:transformation}).  Then there exist
 a unitary quaternion tensor $\mathcal{U}\in\mathbb{H}^{I_{1}\times \cdots \times I_{N}\times I_{1}\times \cdots \times I_{N}}$  such that
\begin{align}
\mathcal{A}=\mathcal{U}*_{N}\mathcal{B}*_{N}\mathcal{U}^{\eta*},
\end{align}
where $\mathcal{B}\in\mathbb{R}^{I_{1}\times \cdots \times I_{N}\times I_{1}\times \cdots \times I_{N}}$ has the following structure
\begin{align}
\mathcal{B}_{i_{1}\cdots  i_{N}j_{1}\cdots j_{N}}=
 \left\{\begin{array}{c}
 d_{i}, (i_{1}\cdots  i_{N}j_{1}\cdots j_{N})=(p_{1}^{i}\cdots  p_{N}^{i}p_{1}^{i}\cdots p_{N}^{i}),\\
 0, \mbox{otherwise},
 \end{array}
  \right.
\end{align}
$d_{i}~ (i=1,\ldots,r)$ are the nonzero singular values of the quaternion matrix $f(\mathcal{A})$ and
\begin{align}
p_{N}^{i}=\left[\frac{i-1}{\prod_{s=1}^{N-1}I_{s}}\right]+1,
~
p_{N-1}^{i}=\left[\frac{i-1-(p_{N}^{i}-1)\prod_{s=1}^{N-1}I_{s}}{\prod_{s=1}^{N-2}I_{s}}\right]+1,
\end{align}
\begin{align}
\vdots
\end{align}
\begin{align}
p_{t}^{i}=\left[\frac{i-1-\sum_{k=t+1}^{N}(p_{k}^{i}-1)\prod_{s=1}^{k-1}I_{s}}{\prod_{s=1}^{t-1}I_{s}}\right]+1,
\end{align}
\begin{align}
\vdots
\end{align}
\begin{align}
p_{2}^{i}=\left[\frac{i-1-\sum_{k=3}^{N}(p_{k}^{i}-1)\prod_{s=1}^{k-1}I_{s}}{I_{1}}\right]+1,
~
p_{1}^{i}=i-\sum_{k=2}^{N}(p_{k}^{i}-1)\prod_{s=1}^{k-1}I_{s}.
\end{align}
\end{theorem}

\begin{proof}
Let $A=f(\mathcal{A}).$ It follows from Lemma \ref{14lemma35} that
\begin{align}
A=U\begin{pmatrix}D_{r}&0\\0&0\end{pmatrix} U^{\eta*},
\end{align}where $U$ is a unitary quaternion matrix and $D_{r}=\mbox{diag} (d_{1},\ldots,d_{r})$ and $d_{i}~ (i=1,\ldots,r)$ are the nonzero singular values of the quaternion matrix $A$. From the property of $f$ in (\ref{21equ211}) and Lemma \ref{14lemma35}, we have
\begin{align}
f^{-1}(A)=f^{-1}\left(U\begin{pmatrix}D_{r}&0\\0&0\end{pmatrix}U^{\eta*} \right)=f^{-1}(U)*_{N}\left(f^{-1}\begin{pmatrix}D_{r}&0\\0&0\end{pmatrix}\right)*_{N}f^{-1}(U^{\eta*})
\end{align}
\begin{align}
\Longrightarrow
\mathcal{A}=\mathcal{U}*_{N}\mathcal{B}*_{N}\mathcal{U}^{\eta*},
\end{align}
where $\mathcal{U}=f^{-1}(U)$ and $\mathcal{B}=f^{-1}\begin{pmatrix}D_{r}&0\\0&0\end{pmatrix}$ is a real tensor whose nonzero entries are $d_{i}$. Note that
\begin{align}
\mathcal{U}*_{N}\mathcal{U}^{*}=f^{-1}(U)*_{N}f^{-1}(U)^{*}=f^{-1}(U)*_{N}f^{-1}(U^{*})=f^{-1}(UU^{*})=f^{-1}(I)=\mathcal{I},
\end{align}
\begin{align}
\mathcal{U}^{*}*_{N}\mathcal{U}=f^{-1}(U)^{*}*_{N}f^{-1}(U)=f^{-1}(U^{*})*_{N}f^{-1}(U)=f^{-1}(U^{*}U)=f^{-1}(I)=\mathcal{I}.
\end{align}
Hence, $\mathcal{U}$ is a unitary quaternion tensor. We can give the subscript of $d_{i}$ in the tensor $\mathcal{B}$ by using the method that in Theorem \ref{23theorem32}.
\end{proof}

\begin{example}
Consider the decomposition for the $2\times 2\times 2\times 2$-quaternion tensor $\mathcal{A}$ defined by
\begin{align*}
a_{1111}=1+\frac{1+\sqrt{3}}{4}\mathbf{i},~a_{1121}=-1+\frac{1+\sqrt{3}}{4}\mathbf{i},~a_{1112}=-\frac{1}{4}+\frac{\sqrt{3}}{4},
~a_{1122}=-\frac{1}{4}+\frac{\sqrt{3}}{4},
\end{align*}
\begin{align*}
a_{2111}=-1+\frac{1+\sqrt{3}}{4}\mathbf{i},~a_{2121}=1+\frac{1+\sqrt{3}}{4}\mathbf{i},~a_{2112}=-\frac{1}{4}+\frac{\sqrt{3}}{4},
~a_{2122}=-\frac{1}{4}+\frac{\sqrt{3}}{4},
\end{align*}
\begin{align*}
a_{1211}=-\frac{1}{4}+\frac{\sqrt{3}}{4},~a_{1221}=-\frac{1}{4}+\frac{\sqrt{3}}{4},
~a_{1212}=1-\frac{1+\sqrt{3}}{4}\mathbf{i},~a_{1222}=-1-\frac{1+\sqrt{3}}{4}\mathbf{i},
\end{align*}
\begin{align*}
a_{2211}=-\frac{1}{4}+\frac{\sqrt{3}}{4},~a_{2221}=-\frac{1}{4}+\frac{\sqrt{3}}{4},
~
a_{2212}=-1-\frac{1+\sqrt{3}}{4}\mathbf{i},
~a_{2222}=1-\frac{1+\sqrt{3}}{4}\mathbf{i}.
\end{align*}
Note that $\mathcal{A}$ is $\mathbf{k}$-Hermitian.
Then we have
\begin{align}
\mathcal{A}=\mathcal{U}*_{2}\mathcal{B}*_{2}\mathcal{U}^{\mathbf{k}*},
\end{align}where $\mathcal{U}\in \mathbb{H}^{2\times 2\times 2\times 2}$ is unitary quaternion tensor, $\mathcal{B}\in \mathbb{R}^{2\times 2\times 2\times 2}$ is a real tensor and
\begin{align*}
\mathcal{B}_{i_{1}i_{2}j_{1}j_{2}}=
 \left\{\begin{array}{c}
 1, ~\mbox{if}~(i_{1}i_{2}j_{1}j_{2})=(1111),\\
 \sqrt{3}, ~\mbox{if}~(i_{1}i_{2}j_{1}j_{2})=(2121),\\
 2, ~\mbox{if}~(i_{1}i_{2}j_{1}j_{2})=(1212),\\
 3, ~\mbox{if}~(i_{1}i_{2}j_{1}j_{2})=(2222),\\
 0,~\mbox{otherwise},
 \end{array}
  \right.
\end{align*}
\begin{align*}
\mathcal{U}_{1111}=\frac{\sqrt{2}}{4}(1+\mathbf{i}),~\mathcal{U}_{1121}=-\frac{\sqrt{2}}{4}(1+\mathbf{i}),
~\mathcal{U}_{1112}=\frac{\sqrt{2}}{2}\mathbf{k},~\mathcal{U}_{1122}=0,
\end{align*}
\begin{align*}
\mathcal{U}_{2111}=\frac{\sqrt{2}}{4}(1+\mathbf{i}),~\mathcal{U}_{2121}=\frac{\sqrt{2}}{4}(1+\mathbf{i}),
~\mathcal{U}_{2112}=-\frac{\sqrt{2}}{2}\mathbf{k},~\mathcal{U}_{2122}=0,
\end{align*}
\begin{align*}
\mathcal{U}_{1211}=\frac{\sqrt{2}}{4}(-1+\mathbf{i}),~\mathcal{U}_{1221}=\frac{\sqrt{2}}{4}(-1+\mathbf{i}),
~\mathcal{U}_{1212}=0,~\mathcal{U}_{1222}=-\frac{\sqrt{2}}{2}\mathbf{k},
\end{align*}
\begin{align*}
\mathcal{U}_{2211}=\frac{\sqrt{2}}{4}(-1+\mathbf{i}),~\mathcal{U}_{2221}=\frac{\sqrt{2}}{4}(-1+\mathbf{i}),
~
\mathcal{U}_{2212}=0,
~\mathcal{U}_{2222}=\frac{\sqrt{2}}{2}\mathbf{k}.
\end{align*}

\end{example}

\section{\textbf{Moore-Penrose inverse of quaternion tensors via Einstein product}}

In this section, we consider the Moore-Penrose inverse of quaternion tensors by using the quaternion tensor SVD. At first, we define the Moore-Penrose inverse of quaternion tensor via Einstein product. The definition of the Moore-Penrose inverse of tensors over the complex number field was given in \cite{Sun}.

\begin{definition}\label{23defi41}
Let $\mathcal{A}\in\mathbb{H}^{I_{1}\times \cdots \times I_{N}\times J_{1}\times \cdots \times J_{N}}$. The tensor
$\mathcal{X}\in\mathbb{H}^{J_{1}\times \cdots \times J_{N}\times I_{1}\times \cdots \times I_{N}}$ satisfying the following four quaternion tensor equations\\
$(1)$ $\mathcal{A}*_{N}\mathcal{X}*_{N}\mathcal{A}=\mathcal{A}$;\\
$(2)$ $\mathcal{X}*_{N}\mathcal{A}*_{N}\mathcal{X}=\mathcal{X}$;\\
$(3)$ $(\mathcal{A}*_{N}\mathcal{X})^{*}=\mathcal{A}*_{N}\mathcal{X}$;\\
$(4)$ $(\mathcal{X}*_{N}\mathcal{A})^{*}=\mathcal{X}*_{N}\mathcal{A}$,\\
is called the Moore-Penrose inverse of $\mathcal{A}$, and is denoted by $\mathcal{A}^{\dag}$.
\end{definition}
Now we give the Moore-Penrose inverse of a quaternion tensor.

\begin{theorem}\label{03theorem41}
Let the SVD of the quaternion tensor $\mathcal{A}$ be
\begin{align}
\mathcal{A}=\mathcal{U}*_{N}\mathcal{B}*_{N}\mathcal{V}^{*},
\end{align}
where $\mathcal{U},\mathcal{V}$ are unitary tensors and the tensor $\mathcal{B}\in\mathbb{R}^{I_{1}\times \cdots \times I_{N}\times J_{1}\times \cdots \times J_{N}}$ is given in (\ref{23equ033}). Then the Moore-Penrose inverse of $\mathcal{A}$ exists and is unique and
\begin{align}
\mathcal{A}^{\dag}=\mathcal{V}*_{N}\mathcal{B}^{\dag}*_{N}\mathcal{U}^{*},
\end{align}
where the tensor $\mathcal{B}^{\dag}\in\mathbb{R}^{J_{1}\times \cdots \times J_{N}\times I_{1}\times \cdots \times I_{N}}$ has the following structure
\begin{align}\label{23equ033}
\mathcal{B}^{\dag}_{j_{1}\cdots  j_{N}i_{1}\cdots i_{N}}=
 \left\{\begin{array}{c}
 d^{-1}_{i}, ~\mbox{if}~(j_{1}\cdots  j_{N}i_{1}\cdots i_{N})=(q_{1}^{i}\cdots  q_{N}^{i}p_{1}^{i}\cdots p_{N}^{i}),\\
 0, \mbox{otherwise},
 \end{array}
  \right.
\end{align}
and $d_{i}~ (i=1,\ldots,r)$ are the positive singular values of the quaternion matrix $f(\mathcal{A})$ and the expressions of $q_{1}^{i},\ldots,  q_{N}^{i},p_{1}^{i},\ldots, p_{N}^{i}$ are given in (\ref{23equ034})-(\ref{23equ313}).
\end{theorem}

\begin{proof}
Note that
\begin{align}
\mathcal{A}^{\dag}=\mathcal{V}*_{N}\mathcal{B}^{\dag}*_{N}\mathcal{U}^{*},
\end{align}satisfy Definition \ref{23defi41}. Thus, the Moore-Penrose inverse of $\mathcal{A}$ exists. Now we want to prove the Moore-Penrose inverse is unique. Let $\mathcal{X}$ and $\mathcal{Y}$ be Moore-Penrose inverses of $\mathcal{A}$. Then we have
\begin{align*}
\mathcal{X}&=\mathcal{X}*_{N}(\mathcal{A}*_{N}\mathcal{X})^{*}=\mathcal{X}*_{N}\mathcal{X}^{*}*_{N}\mathcal{A}^{*}
=\mathcal{X}*_{N}(\mathcal{A}*_{N}\mathcal{X})^{*}*_{N}(\mathcal{A}*_{N}\mathcal{Y})^{*}\\&
=\mathcal{X}*_{N}\mathcal{A}*_{N}\mathcal{Y}=(\mathcal{X}*_{N}\mathcal{A})^{*}*_{N}
(\mathcal{Y}*_{N}\mathcal{A})^{*}*_{N}\mathcal{Y}=\mathcal{A}^{*}*_{N}\mathcal{Y}^{*}*_{N}\mathcal{Y}\\
&=(\mathcal{Y}*_{N}\mathcal{A})^{*}*_{N}\mathcal{Y}=\mathcal{Y}.
\end{align*}

\end{proof}

\begin{remark}
The Moore-Penrose inverse of tensor $\mathcal{A}$ over the complex number field was discussed in \cite{Sun}. Theorem \ref{03theorem41} not only extend it to the quaternion algebra, but also give the positions of $d_{i}^{-1}$ in the real tensor $\mathcal{B}^{\dag}$. Note that the real tensor $\mathcal{B}^{\dag}$ does not necessarily satisfy the following condition:
\begin{align}
\mathcal{B}^{\dag}_{j_{1}\cdots  j_{N}i_{1}\cdots i_{N}}=
 0~~\mbox{if}~~(j_{1}\cdots  j_{N})\neq(i_{1}\cdots i_{N}).
\end{align}
\end{remark}

\begin{example}
The Moore-Penrose inverse of the quaternion tensor $\mathcal{A}$ in Example \ref{27example2} can be expressed as
\begin{align*}
\mathcal{A}^{\dag}=\mathcal{V}*_{2}\mathcal{B}^{\dag}*_{2}\mathcal{U}^{*},
\end{align*}where the unitary quaternion tensors $\mathcal{V}$ and $\mathcal{U}$ are given in Example \ref{27example2} and
$\mathcal{B}^{\dag}\in \mathbb{R}^{3\times 2\times 2\times 2}$ is a real tensor and
\begin{align*}
\mathcal{B}^{\dag}_{j_{1}j_{2}i_{1}i_{2}}=
 \left\{\begin{array}{c}
 1, ~\mbox{if}~(j_{1}j_{2}i_{1}i_{2})=(1111),\\
 \frac{1}{\sqrt{2}}, ~\mbox{if}~(j_{1}j_{2}i_{1}i_{2})=(2121),\\
 \frac{1}{2}, ~\mbox{if}~(j_{1}j_{2}i_{1}i_{2})=(3112),\\
 \frac{1}{4}, ~\mbox{if}~(j_{1}j_{2}i_{1}i_{2})=(1222),\\
 0,~\mbox{otherwise}.
 \end{array}
  \right.
\end{align*}
\end{example}

It is easy to obtain the following results.
\begin{proposition}
 For the quaternion tensor $\mathcal{A}\in\mathbb{H}^{I_{1}\times \cdots \times I_{N}\times J_{1}\times \cdots \times J_{N}}$, the symbols $\mathcal{L}_{\mathcal{A}}$ and $\mathcal{R}_{\mathcal{A}}$ stand for
\begin{align*}
\mathcal{L}_{\mathcal{A}}=\mathcal{I}-\mathcal{A}^{\dag}*_{N}\mathcal{A},\quad
\mathcal{R}_{\mathcal{A}}=\mathcal{I}-\mathcal{A}*_{N}\mathcal{A}^{\dag}.
\end{align*}Then\\
$(1)$ $(\mathcal{A}^{*})^{\dag}=(\mathcal{A}^{\dag})^{*}$;\\
$(2)$ $(\mathcal{A}^{\eta*})^{\dag}=(\mathcal{A}^{\dag})^{\eta*}$;\\
$(3)$ $\mathcal{L}_{\mathcal{A}}=\mathcal{L}^{*}_{\mathcal{A}}$, $\mathcal{R}_{\mathcal{A}}=\mathcal{R}^{*}_{\mathcal{A}}$;\\
$(4)$ $\mathcal{L}^{\eta*}_{\mathcal{A}}=\mathcal{R}_{\mathcal{A}^{\eta*}}$, $\mathcal{R}^{\eta*}_{\mathcal{A}}=\mathcal{L}_{\mathcal{A}^{\eta*}}$.
\end{proposition}

\section{\textbf{The general solution to the quaternion tensor equation (\ref{tensorequ01})}}

Our goal of this section is to give some solvability conditions for the generalized Sylvester quaternion tensor equation (\ref{tensorequ01}) to posses a solution and to provide an expression of this general solution when the solvability conditions are met. As an application of the quaternion tensor equation (\ref{tensorequ01}), we can derive some solvability conditions and the general $\eta$-Hermitian solution to a quaternion tensor equation involving $\eta$-Hermicity, i.e.,
\begin{align}\label{htensorequ01}
\mathcal{A}*_{N}\mathcal{X}*_{N}\mathcal{A}^{\eta*}+\mathcal{C}*_{N}\mathcal{Y}*_{N}\mathcal{C}^{\eta*}=\mathcal{E},
\quad \mathcal{X}=\mathcal{X}^{\eta*},\quad \mathcal{Y}=\mathcal{Y}^{\eta*},
\end{align}where $\mathcal{A},~\mathcal{C}$ and $\mathcal{E}=\mathcal{E}^{\eta*}$ are given quaternion tensors with suitable order.

Now we give the main theorem of this section.
\begin{theorem}\label{theorem01}
Let $\mathcal{A}\in\mathbb{H}^{I_{1}\times \cdots \times I_{N}\times J_{1}\times \cdots \times J_{N}},
\mathcal{B}\in\mathbb{H}^{K_{1}\times \cdots \times K_{M}\times L_{1}\times \cdots \times L_{M}},
\mathcal{C}\in\mathbb{H}^{I_{1}\times \cdots \times I_{N}\times G_{1}\times \cdots \times G_{N}},$ $
\mathcal{D}\in\mathbb{H}^{H_{1}\times \cdots \times H_{M}\times L_{1}\times \cdots \times L_{M}},$ and
$\mathcal{E}\in\mathbb{H}^{I_{1}\times \cdots \times I_{N}\times L_{1}\times \cdots \times L_{M}}$. Set
\begin{align}
\mathcal{P}=(\mathcal{R_{A}})*_{N}\mathcal{C},\quad  \mathcal{Q}=\mathcal{D}*_{M}\mathcal{L_{B}},\quad \mathcal{S}=\mathcal{C}*_{N}\mathcal{L_{P}}.
\end{align}
Then the generalized Sylvester quaternion tensor equation (\ref{tensorequ01}) is consistent if and only if
\begin{align}\label{25equ022}
(\mathcal{R}_{\mathcal{P}})*_{N}(\mathcal{R}_{\mathcal{A}})*_{N}\mathcal{E}=0,
\mathcal{E}*_{M}(\mathcal{L}_{\mathcal{B}})*_{M}(\mathcal{L}_{\mathcal{Q}})=0,
(\mathcal{R}_{\mathcal{A}})*_{N}\mathcal{E}*_{M}\mathcal{L_{D}}=0,
(\mathcal{R}_{\mathcal{C}})*_{N}\mathcal{E}*_{M}\mathcal{L_{B}}=0.
\end{align}
In this case, the general solution to (\ref{tensorequ01}) can be
expressed as
\begin{align}\label{25equ023}
\mathcal{X}=&\mathcal{A}^{\dag}*_{N}\mathcal{E}*_{M}\mathcal{B}^{\dag}-
\mathcal{A}^{\dag}*_{N}\mathcal{C}*_{N}\mathcal{P}^{\dag}*_{N}\mathcal{E}*_{M}\mathcal{B}^{\dag}
-\mathcal{A}^{\dag}*_{N}\mathcal{S}*_{N}\mathcal{C}^{\dag}*_{N}\mathcal{E}*_{M}\mathcal{Q}^{\dag}*_{M}\mathcal{D}*_{M}\mathcal{B}^{\dag}
\nonumber\\
&-\mathcal{A}^{\dag}*_{N}\mathcal{S}*_{N}\mathcal{U}_{2}*_{M}(\mathcal{R_{Q}})*_{M}\mathcal{D}*_{M}\mathcal{B}^{\dag}
+(\mathcal{L_{A}})*_{N}\mathcal{U}_{4}+\mathcal{U}_{5}*_{M}\mathcal{R_{B}},
\end{align}
\begin{align}\label{25equ024}
\mathcal{Y}=&\mathcal{P}^{\dag}*_{N}\mathcal{E}*_{M}\mathcal{D}^{\dag}+
\mathcal{S}^{\dag}*_{N}\mathcal{S}*_{N}\mathcal{C}^{\dag}*_{N}\mathcal{E}*_{M}\mathcal{Q}^{\dag}
+(\mathcal{L_{P}})*_{N}(\mathcal{L_{S}})*_{N}\mathcal{U}_{1}\nonumber\\&
+(\mathcal{L_{P}})*_{N}\mathcal{U}_{2}*_{M}\mathcal{R_{Q}}+
\mathcal{U}_{3}*_{M}\mathcal{R_{D}},
\end{align}
where $\mathcal{U}_{1},\mathcal{U}_{2},\mathcal{U}_{3},\mathcal{U}_{4},\mathcal{U}_{5}$ are arbitrary quaternion tensors with suitable order.
\end{theorem}

\begin{proof}
$\Longrightarrow:$ If the generalized Sylvester quaternion tensor equation (\ref{tensorequ01}) has a
solution, say $(\mathcal{X}^{1},\mathcal{Y}^{1})$, then we have
\begin{align}\label{25equ025}
\mathcal{A}*_{N}\mathcal{X}^{1}*_{M}\mathcal{B}+\mathcal{C}*_{N}\mathcal{Y}^{1}*_{M}\mathcal{D}=\mathcal{E}.
\end{align}It follows from
\begin{align}
(\mathcal{R}_{\mathcal{A}})*_{N}\mathcal{A}=0,
~\mathcal{B}*_{M}\mathcal{L}_{\mathcal{B}}=0,~
(\mathcal{R}_{\mathcal{C}})*_{N}\mathcal{C}=0,~
\mathcal{D}*_{M}\mathcal{L_{D}}=0,
\end{align}
\begin{align}
~\mathcal{P}=(\mathcal{R_{A}})*_{N}C,
~\mathcal{Q}=\mathcal{D}*_{M}\mathcal{L_{B}},
~(\mathcal{R}_{\mathcal{P}})*_{N}\mathcal{P}=0,~
\mathcal{Q}*_{M}\mathcal{L}_{\mathcal{Q}}=0,
\end{align}
and (\ref{25equ025}) that
\begin{align}
(\mathcal{R}_{\mathcal{P}})*_{N}(\mathcal{R}_{\mathcal{A}})*_{N}\mathcal{E}&=
(\mathcal{R}_{\mathcal{P}})*_{N}(\mathcal{R}_{\mathcal{A}})*_{N}
(\mathcal{A}*_{N}\mathcal{X}^{1}*_{M}\mathcal{B}+\mathcal{C}*_{N}\mathcal{Y}^{1}*_{M}\mathcal{D})\nonumber
\\&=(\mathcal{R}_{\mathcal{P}})*_{N}(\mathcal{R}_{\mathcal{A}})*_{N}
\mathcal{A}*_{N}\mathcal{X}^{1}*_{N}\mathcal{A}^{*}+
(\mathcal{R}_{\mathcal{P}})*_{N}(\mathcal{R}_{\mathcal{A}})*_{N}\mathcal{C}*_{N}\mathcal{Y}^{1}*_{N}\mathcal{C}^{*}
\nonumber
\\&=(\mathcal{R}_{\mathcal{P}})*_{N}(\mathcal{R}_{\mathcal{A}})*_{N}
\mathcal{A}*_{N}\mathcal{X}^{1}*_{N}\mathcal{A}^{*}+
(\mathcal{R}_{\mathcal{P}})*_{N}\mathcal{P}*_{N}\mathcal{Y}^{1}*_{N}\mathcal{C}^{*}\nonumber\\
&=0,
\end{align}
\begin{align}
\mathcal{E}*_{M}(\mathcal{L}_{\mathcal{B}})*_{M}(\mathcal{L}_{\mathcal{Q}})&=
(\mathcal{A}*_{N}\mathcal{X}^{1}*_{M}\mathcal{B}+\mathcal{C}*_{N}\mathcal{Y}^{1}*_{M}\mathcal{D})*_{M}
(\mathcal{L}_{\mathcal{B}})*_{M}(\mathcal{L}_{\mathcal{Q}})\nonumber\\
&=\mathcal{A}*_{N}\mathcal{X}^{1}*_{M}\mathcal{B}*_{M}
(\mathcal{L}_{\mathcal{B}})*_{M}(\mathcal{L}_{\mathcal{Q}})+
\mathcal{C}*_{N}\mathcal{Y}^{1}*_{M}\mathcal{D}*_{M}
(\mathcal{L}_{\mathcal{B}})*_{M}(\mathcal{L}_{\mathcal{Q}})\nonumber\\
&=\mathcal{A}*_{N}\mathcal{X}^{1}*_{M}\mathcal{B}*_{M}
(\mathcal{L}_{\mathcal{B}})*_{M}(\mathcal{L}_{\mathcal{Q}})+
\mathcal{C}*_{N}\mathcal{Y}^{1}*_{M}\mathcal{Q}*_{M}(\mathcal{L}_{\mathcal{Q}})\nonumber\\
&=0,
\end{align}
\begin{align}
(\mathcal{R}_{\mathcal{A}})*_{N}\mathcal{E}*_{M}\mathcal{L_{D}}&=
(\mathcal{R}_{\mathcal{A}})*_{N}(\mathcal{A}*_{N}\mathcal{X}^{1}*_{M}\mathcal{B}+\mathcal{C}*_{N}\mathcal{Y}^{1}*_{M}\mathcal{D})*_{M}\mathcal{L_{D}}\nonumber\\&=
(\mathcal{R}_{\mathcal{A}})*_{N}\mathcal{A}*_{N}\mathcal{X}^{1}*_{M}\mathcal{B}*_{M}\mathcal{L_{D}}+
\mathcal{C}*_{N}\mathcal{Y}^{1}*_{M}\mathcal{D}*_{M}\mathcal{L_{D}}
\nonumber\\&=0,
\end{align}
\begin{align}
(\mathcal{R}_{\mathcal{C}})*_{N}\mathcal{E}*_{M}\mathcal{L_{B}}&=
(\mathcal{R}_{\mathcal{C}})*_{N}(\mathcal{A}*_{N}\mathcal{X}^{1}*_{M}\mathcal{B}
+\mathcal{C}*_{N}\mathcal{Y}^{1}*_{M}\mathcal{D})*_{M}\mathcal{L_{B}}\nonumber\\
&=(\mathcal{R}_{\mathcal{C}})*_{N}\mathcal{A}*_{N}\mathcal{X}^{1}*_{M}\mathcal{B}*_{M}\mathcal{L_{B}}+
(\mathcal{R}_{\mathcal{C}})*_{N}\mathcal{C}*_{N}\mathcal{Y}^{1}*_{M}\mathcal{D}*_{M}\mathcal{L_{B}}\nonumber\\
&=0.
\end{align}

$\Longleftarrow:$ Now we want to prove that the tensors $\mathcal{X}$ and $\mathcal{Y}$ having the form of (\ref{25equ023}) and (\ref{25equ024}), respectively, are a solution to (\ref{tensorequ01}) under the equalities in (\ref{25equ022}). Substituting (\ref{25equ023}) and (\ref{25equ024}) into (\ref{tensorequ01}) yields
\begin{align}\label{28equ212}
&\mathcal{A}*_{N}\mathcal{X}*_{M}\mathcal{B}+\mathcal{C}*_{N}\mathcal{Y}*_{M}\mathcal{D}\nonumber\\=&
\mathcal{A}*_{N}\mathcal{A}^{\dag}*_{N}\mathcal{E}*_{M}\mathcal{B}^{\dag}*_{M}\mathcal{B}-
\mathcal{A}*_{N}\mathcal{A}^{\dag}*_{N}\mathcal{C}*_{N}\mathcal{P}^{\dag}*_{N}\mathcal{E}*_{M}\mathcal{B}^{\dag}*_{M}\mathcal{B}\nonumber\\&
-\mathcal{A}*_{N}\mathcal{A}^{\dag}*_{N}\mathcal{S}*_{N}\mathcal{C}^{\dag}*_{N}\mathcal{E}*_{M}\mathcal{Q}^{\dag}*_{M}\mathcal{D}*_{M}\mathcal{B}^{\dag}*_{M}\mathcal{B}
\nonumber\\
&-\mathcal{A}*_{N}\mathcal{A}^{\dag}*_{N}\mathcal{S}*_{N}\mathcal{U}_{2}*_{M}(\mathcal{R_{Q}})*_{M}\mathcal{D}*_{M}\mathcal{B}^{\dag}*_{M}\mathcal{B}
+\mathcal{C}*_{N}\mathcal{P}^{\dag}*_{N}\mathcal{E}*_{M}\mathcal{D}^{\dag}*_{M}\mathcal{D}\nonumber\\&+
\mathcal{C}*_{N}\mathcal{S}^{\dag}*_{N}\mathcal{S}*_{N}\mathcal{C}^{\dag}*_{N}\mathcal{E}*_{M}\mathcal{Q}^{\dag}*_{M}\mathcal{D}
+\mathcal{S}*_{N}\mathcal{U}_{2}*_{M}(\mathcal{R_{Q}})*_{M}\mathcal{D}.
\end{align}From
\begin{align}
(\mathcal{R_{A}})*_{N}\mathcal{S}=(\mathcal{R_{A}})*_{N}\mathcal{C}*_{N}\mathcal{L_{P}}=
\mathcal{P}*_{N}\mathcal{L_{P}}=0
\end{align} and
\begin{align}\label{26equ214}
\mathcal{D}*_{M}\mathcal{B}^{\dag}*_{M}\mathcal{B}=
\mathcal{D}-\mathcal{Q},
\end{align}
we see that
\begin{align}\label{26equ215}
\mathcal{A}*_{N}\mathcal{A}^{\dag}*_{N}\mathcal{S}*_{N}=\mathcal{S}*_{N},\qquad
(\mathcal{R_{Q}})*_{M}\mathcal{D}*_{M}\mathcal{B}^{\dag}*_{M}\mathcal{B}=\mathcal{R_{Q}}*_{M}\mathcal{D}.
\end{align}
Hence, we have
\begin{align}\label{28equ216}
&\mathcal{S}*_{N}\mathcal{U}_{2}*_{M}(\mathcal{R_{Q}})*_{M}\mathcal{D}-
\mathcal{A}*_{N}\mathcal{A}^{\dag}*_{N}\mathcal{S}*_{N}\mathcal{U}_{2}*_{M}(\mathcal{R_{Q}})*_{M}\mathcal{D}*_{M}\mathcal{B}^{\dag}*_{M}\mathcal{B}\nonumber\\
=&\mathcal{S}*_{N}\mathcal{U}_{2}*_{M}(\mathcal{R_{Q}})*_{M}\mathcal{D}-
\mathcal{S}*_{N}\mathcal{U}_{2}*_{M}(\mathcal{R_{Q}})*_{M}\mathcal{D}\nonumber\\
=&0.
\end{align}
From
\begin{align}\label{29equ217}
\mathcal{P}^{\dag}*_{N}\mathcal{P}*_{N}\mathcal{S}^{\dag}*_{N}\mathcal{S}&=
(\mathcal{P}^{\dag}*_{N}\mathcal{P})^{*}*_{N}(\mathcal{S}^{\dag}*_{N}\mathcal{S})^{*}\nonumber\\&
=
(\mathcal{S}^{\dag}*_{N}\mathcal{S}*_{N}\mathcal{P}^{\dag}*_{N}\mathcal{P})^{*}\nonumber\\&
=(\mathcal{S}^{\dag}*_{N}\mathcal{C}*_{N}(\mathcal{L_{P}})*_{N}\mathcal{P}^{\dag}*_{N}\mathcal{P})^{*}\nonumber\\&
=0,
\end{align}
we infer that
\begin{align}
\mathcal{C}*_{N}\mathcal{S}^{\dag}*_{N}\mathcal{S}-\mathcal{S}
&=\mathcal{C}*_{N}\mathcal{S}^{\dag}*_{N}\mathcal{S}-\mathcal{S}*_{N}\mathcal{S}^{\dag}*_{N}\mathcal{S}\nonumber\\&
=(\mathcal{C}-\mathcal{S})*_{N}\mathcal{S}^{\dag}*_{N}\mathcal{S}\nonumber\\&
=\mathcal{C}*_{N}\mathcal{P}^{\dag}*_{N}\mathcal{P}*_{N}\mathcal{S}^{\dag}*_{N}\mathcal{S}\nonumber\\&
=0,
\end{align}
i.e.,
\begin{align}\label{26equ219}
\mathcal{C}*_{N}\mathcal{S}^{\dag}*_{N}\mathcal{S}=\mathcal{S}.
\end{align}
Since $\mathcal{Q}*_{M}\mathcal{B}^{\dag}=\mathcal{D}*_{M}(\mathcal{L_{B}})*_{M}\mathcal{B}^{\dag}=0$, we have
\begin{align}\label{26equ220}
\mathcal{B}^{\dag}*_{M}\mathcal{B}*_{M}\mathcal{Q}^{\dag}*_{M}\mathcal{Q}
=(\mathcal{B}^{\dag}*_{M}\mathcal{B})^{*}*_{M}(\mathcal{Q}^{\dag}*_{M}\mathcal{Q})^{*}
=(\mathcal{Q}^{\dag}*_{M}\mathcal{Q}*_{M}\mathcal{B}^{\dag}*_{M}\mathcal{B})^{*}=0.
\end{align}
Thus from (\ref{26equ214}), (\ref{26equ215}), (\ref{26equ219}), (\ref{26equ220}) and $\mathcal{E}*_{M}(\mathcal{L}_{\mathcal{B}})*_{M}(\mathcal{L}_{\mathcal{Q}})=0$, we infer that
\begin{align}\label{28equ221}
&\mathcal{C}*_{N}\mathcal{S}^{\dag}*_{N}\mathcal{S}*_{N}\mathcal{C}^{\dag}*_{N}\mathcal{E}*_{M}\mathcal{Q}^{\dag}*_{M}\mathcal{D}-
\mathcal{A}*_{N}\mathcal{A}^{\dag}*_{N}\mathcal{S}*_{N}\mathcal{C}^{\dag}*_{N}\mathcal{E}*_{M}\mathcal{Q}^{\dag}*_{M}\mathcal{D}*_{M}
\mathcal{B}^{\dag}*_{M}\mathcal{B}\nonumber\\
=&\mathcal{S}*_{N}\mathcal{C}^{\dag}*_{N}\mathcal{E}*_{M}\mathcal{Q}^{\dag}*_{M}\mathcal{D}-
\mathcal{S}*_{N}\mathcal{C}^{\dag}*_{N}\mathcal{E}*_{M}\mathcal{Q}^{\dag}*_{M}\mathcal{D}*_{M}
\mathcal{B}^{\dag}*_{M}\mathcal{B}\nonumber\\
=&\mathcal{S}*_{N}\mathcal{C}^{\dag}*_{N}\mathcal{E}*_{M}\mathcal{Q}^{\dag}*_{M}\mathcal{D}-
\mathcal{S}*_{N}\mathcal{C}^{\dag}*_{N}\mathcal{E}*_{M}\mathcal{Q}^{\dag}*_{M}(\mathcal{D}-\mathcal{Q})\nonumber\\
=&\mathcal{S}*_{N}\mathcal{C}^{\dag}*_{N}\mathcal{E}*_{M}\mathcal{Q}^{\dag}*_{M}\mathcal{Q}\nonumber\\
=&\mathcal{S}*_{N}\mathcal{C}^{\dag}*_{N}\mathcal{E}*_{M}(\mathcal{B}^{\dag}*_{M}\mathcal{B}+\mathcal{L_{B}})*_{M}\mathcal{Q}^{\dag}*_{M}\mathcal{Q}\nonumber\\
=&\mathcal{S}*_{N}\mathcal{C}^{\dag}*_{N}\mathcal{E}*_{M}\mathcal{B}^{\dag}*_{M}\mathcal{B}*_{M}\mathcal{Q}^{\dag}*_{M}\mathcal{Q}
+\mathcal{S}*_{N}\mathcal{C}^{\dag}*_{N}\mathcal{E}*_{M}(\mathcal{L_{B}})*_{M}\mathcal{Q}^{\dag}*_{M}\mathcal{Q}\nonumber\\
=&\mathcal{S}*_{N}\mathcal{C}^{\dag}*_{N}\mathcal{E}*_{M}\mathcal{L_{B}}.
\end{align}
Since $\mathcal{A}^{\dag}*_{N}\mathcal{P}=0,$ we see from $(\mathcal{R}_{\mathcal{A}})*_{N}\mathcal{E}*_{M}\mathcal{L_{D}}=0$ that
\begin{align}
\mathcal{C}*_{N}\mathcal{P}^{\dag}*_{N}\mathcal{E}*_{M}(\mathcal{L_{D}})=&
\mathcal{C}*_{N}\mathcal{P}^{\dag}*_{N}\mathcal{A}*_{N}\mathcal{A}^{\dag}*_{N}\mathcal{E}*_{M}(\mathcal{L_{D}})
\nonumber\\=&\mathcal{C}*_{N}\mathcal{P}^{\dag}*_{N}\mathcal{P}*_{N}\mathcal{P}^{\dag}*_{N}\mathcal{A}*_{N}\mathcal{A}^{\dag}*_{N}\mathcal{E}*_{M}(\mathcal{L_{D}})
\nonumber\\=&\mathcal{C}*_{N}\mathcal{P}^{\dag}*_{N}(\mathcal{P}*_{N}\mathcal{P}^{\dag})^{*}*_{N}(\mathcal{A}*_{N}\mathcal{A}^{\dag})^{*}*_{N}\mathcal{E}*_{M}(\mathcal{L_{D}})
\nonumber\\=&\mathcal{C}*_{N}\mathcal{P}^{\dag}*_{N}(\mathcal{A}*_{N}\mathcal{A}^{\dag}*_{N}\mathcal{P}*_{N}\mathcal{P}^{\dag})^{*}*_{N}\mathcal{E}*_{M}(\mathcal{L_{D}})
\nonumber\\=&0,
\end{align}
that is
\begin{align}\label{28equ223}
\mathcal{C}*_{N}\mathcal{P}^{\dag}*_{N}\mathcal{E}*_{M}\mathcal{D}^{\dag}*_{M}\mathcal{D}=\mathcal{C}*_{N}\mathcal{P}^{\dag}*_{N}\mathcal{E}.
\end{align}
It follows from (\ref{28equ223}) and
\begin{align}
(\mathcal{R}_{\mathcal{P}})*_{N}(\mathcal{R}_{\mathcal{A}})*_{N}\mathcal{E}=0,~(\mathcal{R}_{\mathcal{C}})*_{N}\mathcal{E}*_{M}\mathcal{L_{B}}=0,
\end{align}
\begin{align}
\mathcal{P}=\mathcal{C}-\mathcal{A}*_{N}\mathcal{A}^{\dag}*_{N}\mathcal{C},~
\mathcal{S}=\mathcal{C}*_{N}\mathcal{L_{P}}=\mathcal{C}-\mathcal{C}*_{N}\mathcal{P}^{\dag}*_{N}\mathcal{P},
\end{align}
\begin{align}\label{29equ225}
\mathcal{P}^{\dag}*_{N}\mathcal{A}&=\mathcal{P}^{\dag}*_{N}\mathcal{P}*_{N}\mathcal{P}^{\dag}*_{N}\mathcal{A}*_{N}\mathcal{A}^{\dag}*_{N}\mathcal{A}
\nonumber\\&=\mathcal{P}^{\dag}*_{N}(\mathcal{P}*_{N}\mathcal{P}^{\dag})^{*}*_{N}(\mathcal{A}*_{N}\mathcal{A}^{\dag})^{*}*_{N}\mathcal{A}\nonumber\\&
=\mathcal{P}^{\dag}*_{N}(\mathcal{A}*_{N}\mathcal{A}^{\dag}*_{N}\mathcal{P}*_{N}\mathcal{P}^{\dag})^{*}*_{N}\mathcal{A}\nonumber\\&=0,
\end{align}
that
\begin{align}\label{28equ226}
&\mathcal{C}*_{N}\mathcal{P}^{\dag}*_{N}\mathcal{E}*_{M}\mathcal{D}^{\dag}*_{M}\mathcal{D}-
\mathcal{A}*_{N}\mathcal{A}^{\dag}*_{N}\mathcal{C}*_{N}\mathcal{P}^{\dag}*_{N}\mathcal{E}*_{M}\mathcal{B}^{\dag}*_{M}\mathcal{B}\nonumber\\
=&\mathcal{C}*_{N}\mathcal{P}^{\dag}*_{N}\mathcal{E}+(\mathcal{P-C})*_{N}\mathcal{P}^{\dag}*_{N}\mathcal{E}*_{M}\mathcal{B}^{\dag}*_{M}\mathcal{B}\nonumber\\
=&\mathcal{C}*_{N}\mathcal{P}^{\dag}*_{N}\mathcal{E}+\mathcal{P}*_{N}\mathcal{P}^{\dag}*_{N}\mathcal{E}*_{M}\mathcal{B}^{\dag}*_{M}\mathcal{B}-
\mathcal{C}*_{N}\mathcal{P}^{\dag}*_{N}\mathcal{E}*_{M}\mathcal{B}^{\dag}*_{M}\mathcal{B}\nonumber\\
=&\mathcal{P}*_{N}\mathcal{P}^{\dag}*_{N}\mathcal{E}*_{M}\mathcal{B}^{\dag}*_{M}\mathcal{B}+
\mathcal{C}*_{N}\mathcal{P}^{\dag}*_{N}\mathcal{E}*_{M}(\mathcal{L_{B}})\nonumber\\
=&\mathcal{P}*_{N}\mathcal{P}^{\dag}*_{N}(\mathcal{R_{A}}+\mathcal{A*_{N}A^{\dag}*_{N}})*_{N}\mathcal{E}*_{M}\mathcal{B}^{\dag}*_{M}\mathcal{B}+
\mathcal{C}*_{N}\mathcal{P}^{\dag}*_{N}\mathcal{C}*_{N}\mathcal{C}^{\dag}*_{N}\mathcal{E}*_{M}(\mathcal{L_{B}})\nonumber\\
=&\mathcal{P}*_{N}\mathcal{P}^{\dag}*_{N}(\mathcal{R_{A}})*_{N}\mathcal{E}*_{M}\mathcal{B}^{\dag}*_{M}\mathcal{B}
+\mathcal{P}*_{N}\mathcal{P}^{\dag}*_{N}\mathcal{A*_{N}A^{\dag}*_{N}}\mathcal{E}*_{M}\mathcal{B}^{\dag}*_{M}\mathcal{B}\nonumber\\
&+\mathcal{C}*_{N}\mathcal{P}^{\dag}*_{N}\mathcal{C}*_{N}\mathcal{C}^{\dag}*_{N}\mathcal{E}*_{M}(\mathcal{L_{B}})-
\mathcal{C}*_{N}\mathcal{P}^{\dag}*_{N}\mathcal{A}*_{N}\mathcal{A}^{\dag}*_{N}\mathcal{C}*_{N}\mathcal{C}^{\dag}*_{N}\mathcal{E}*_{M}(\mathcal{L_{B}})\nonumber\\
=&\mathcal{P}*_{N}\mathcal{P}^{\dag}*_{N}(\mathcal{R_{A}})*_{N}\mathcal{E}*_{M}\mathcal{B}^{\dag}*_{M}\mathcal{B}+
\mathcal{C}*_{N}\mathcal{P}^{\dag}*_{N}(\mathcal{R_{A}})*_{N}\mathcal{C}*_{N}\mathcal{C}^{\dag}*_{N}\mathcal{E}*_{M}(\mathcal{L_{B}})\nonumber\\
=&(\mathcal{R_{A}})*_{N}\mathcal{E}*_{M}\mathcal{B}^{\dag}*_{M}\mathcal{B}+
\mathcal{C}*_{N}\mathcal{P}^{\dag}*_{N}\mathcal{P}*_{N}\mathcal{C}^{\dag}*_{N}\mathcal{E}*_{M}(\mathcal{L_{B}})\nonumber\\
=&(\mathcal{R_{A}})*_{N}\mathcal{E}*_{M}\mathcal{B}^{\dag}*_{M}\mathcal{B}+
(\mathcal{C}-\mathcal{S})*_{N}\mathcal{C}^{\dag}*_{N}\mathcal{E}*_{M}(\mathcal{L_{B}})\nonumber\\
=&(\mathcal{R_{A}})*_{N}\mathcal{E}*_{M}\mathcal{B}^{\dag}*_{M}\mathcal{B}+
\mathcal{E}*_{M}(\mathcal{L_{B}})-\mathcal{S}*_{N}\mathcal{C}^{\dag}*_{N}\mathcal{E}*_{M}(\mathcal{L_{B}}).
\end{align}
Thus from (\ref{28equ212}), (\ref{28equ216}), (\ref{28equ221}) and (\ref{28equ226}), we have
\begin{align}
&\mathcal{A}*_{N}\mathcal{X}*_{M}\mathcal{B}+\mathcal{C}*_{N}\mathcal{Y}*_{M}\mathcal{D}\nonumber\\=&
\mathcal{A}*_{N}\mathcal{A}^{\dag}*_{N}\mathcal{E}*_{M}\mathcal{B}^{\dag}*_{M}\mathcal{B}+
\mathcal{S}*_{N}\mathcal{C}^{\dag}*_{N}\mathcal{E}*_{M}\mathcal{L_{B}}
\nonumber\\&+(\mathcal{R_{A}})*_{N}\mathcal{E}*_{M}\mathcal{B}^{\dag}*_{M}\mathcal{B}+
\mathcal{E}*_{M}(\mathcal{L_{B}})-\mathcal{S}*_{N}\mathcal{C}^{\dag}*_{N}\mathcal{E}*_{M}(\mathcal{L_{B}})\nonumber\\=&
\mathcal{E}+\mathcal{A}*_{N}\mathcal{A}^{\dag}*_{N}\mathcal{E}*_{M}\mathcal{B}^{\dag}*_{M}\mathcal{B}-
\mathcal{A}*_{N}\mathcal{A}^{\dag}*_{N}\mathcal{E}*_{M}\mathcal{B}^{\dag}*_{M}\mathcal{B}
\nonumber\\&+\mathcal{E}*_{M}\mathcal{B}^{\dag}*_{M}\mathcal{B}-
\mathcal{E}*_{M}\mathcal{B}^{\dag}*_{M}\mathcal{B}
\nonumber\\=&\mathcal{E}.
\end{align}

Now we want to show that for any arbitrary solution, say $(\mathcal{X}^{0},\mathcal{Y}^{0})$ of the generalized Sylvester quaternion tensor equation (\ref{tensorequ01}) can be expressed as (\ref{25equ023}) and (\ref{25equ024}) under the equalities in (\ref{25equ022}). That is to say, we prove that there exist tensors $\mathcal{U}_{1},\mathcal{U}_{2},\mathcal{U}_{3},\mathcal{U}_{4},\mathcal{U}_{5}$ such that (\ref{25equ023}) and (\ref{25equ024}) hold if we substitute $\mathcal{X}$ and $\mathcal{Y}$ by $\mathcal{X}^{0}$ and $\mathcal{Y}^{0}$, respectively. Note that
\begin{align}
\mathcal{A}*_{N}\mathcal{X}^{0}*_{M}\mathcal{B}+\mathcal{C}*_{N}\mathcal{Y}^{0}*_{M}\mathcal{D}=\mathcal{E},
\end{align}
\begin{align}\label{29equ229}
 \mathcal{P}*_{N}\mathcal{Y}^{0}*_{M}\mathcal{D}=(\mathcal{R_{A}})*_{N}\mathcal{C}*_{N}\mathcal{Y}^{0}*_{M}\mathcal{D}=(\mathcal{R_{A}})*_{N}\mathcal{E},
 \end{align}
\begin{align}\label{29equ230}
\mathcal{C}*_{N}\mathcal{Y}^{0}*_{M}\mathcal{Q}=\mathcal{C}*_{N}\mathcal{Y}^{0}*_{M}\mathcal{D}*_{M}\mathcal{(L_{B})}
=\mathcal{E}*_{M}\mathcal{(L_{B})}.
\end{align}
Put
\begin{align}
\mathcal{U}_{1}=\mathcal{Y}^{0}*_{M}\mathcal{Q}*_{M}\mathcal{Q}^{\dag},~
\mathcal{U}_{2}=\mathcal{Y}^{0}*_{M}\mathcal{D}*_{M}\mathcal{D}^{\dag},~
\mathcal{U}_{3}=\mathcal{Y}^{0},~\mathcal{U}_{4}=\mathcal{X}^{0}*_{M}\mathcal{B}*_{M}\mathcal{B}^{\dag},~\mathcal{U}_{5}=\mathcal{X}^{0}.
\end{align}
From (\ref{29equ225}), (\ref{29equ229}),
\begin{align}
\mathcal{D}*_{M}\mathcal{D}^{\dag}*_{M}\mathcal{Q}*_{M}\mathcal{Q}^{\dag}=
\mathcal{D}*_{M}\mathcal{D}^{\dag}*_{M}\mathcal{D}*_{M}(\mathcal{L_{B}})*_{M}\mathcal{Q}^{\dag}=\mathcal{Q}*_{M}\mathcal{Q}^{\dag},
\end{align}
and
\begin{align}
(\mathcal{L_{P}})*_{N}\mathcal{S}^{\dag}*_{N}\mathcal{S}=(\mathcal{L_{P}})*_{N}(\mathcal{S}^{\dag}*_{N}\mathcal{S})^{*}=
(\mathcal{L_{P}})*_{N}\mathcal{S}^{*}*_{N}(\mathcal{S}^{\dag})^{*}=\mathcal{S}^{*}*_{N}(\mathcal{S}^{\dag})^{*}=\mathcal{S}^{\dag}*_{N}\mathcal{S},
\end{align}
we infer that
\begin{align}\label{29equ234}
\mathcal{Y}=&\mathcal{P}^{\dag}*_{N}\mathcal{E}*_{M}\mathcal{D}^{\dag}+
\mathcal{S}^{\dag}*_{N}\mathcal{S}*_{N}\mathcal{C}^{\dag}*_{N}\mathcal{E}*_{M}\mathcal{Q}^{\dag}
+(\mathcal{L_{P}})*_{N}(\mathcal{L_{S}})*_{N}\mathcal{Y}^{0}*_{M}\mathcal{Q}*_{M}\mathcal{Q}^{\dag}
\nonumber\\&+(\mathcal{L_{P}})*_{N}\mathcal{Y}^{0}*_{M}\mathcal{D}*_{M}\mathcal{D}^{\dag}*_{M}\mathcal{R_{Q}}+
\mathcal{Y}^{0}*_{M}\mathcal{R_{D}}\nonumber\\
=&\mathcal{P}^{\dag}*_{N}\mathcal{E}*_{M}\mathcal{D}^{\dag}+
\mathcal{S}^{\dag}*_{N}\mathcal{S}*_{N}\mathcal{C}^{\dag}*_{N}\mathcal{E}*_{M}\mathcal{Q}^{\dag}
+(\mathcal{L_{P}})*_{N}\mathcal{Y}^{0}*_{M}\mathcal{Q}*_{M}\mathcal{Q}^{\dag}
\nonumber\\
&-\mathcal{S}^{\dag}*_{N}\mathcal{S}*_{N}\mathcal{Y}^{0}*_{M}\mathcal{Q}*_{M}\mathcal{Q}^{\dag}
+(\mathcal{L_{P}})*_{N}\mathcal{Y}^{0}*_{M}\mathcal{D}*_{M}\mathcal{D}^{\dag}
\nonumber\\&-(\mathcal{L_{P}})*_{N}\mathcal{Y}^{0}*_{M}\mathcal{D}*_{M}\mathcal{D}^{\dag}*_{M}\mathcal{Q}*_{M}\mathcal{Q}^{\dag}+
\mathcal{Y}^{0}-\mathcal{Y}^{0}*_{M}\mathcal{D}*_{M}\mathcal{D}^{\dag}\nonumber\\
=&\mathcal{P}^{\dag}*_{N}\mathcal{E}*_{M}\mathcal{D}^{\dag}+
\mathcal{S}^{\dag}*_{N}\mathcal{S}*_{N}\mathcal{C}^{\dag}*_{N}\mathcal{E}*_{M}\mathcal{Q}^{\dag}
+(\mathcal{L_{P}})*_{N}\mathcal{Y}^{0}*_{M}\mathcal{Q}*_{M}\mathcal{Q}^{\dag}
\nonumber\\
&-\mathcal{S}^{\dag}*_{N}\mathcal{S}*_{N}\mathcal{Y}^{0}*_{M}\mathcal{Q}*_{M}\mathcal{Q}^{\dag}
+\mathcal{Y}^{0}*_{M}\mathcal{D}*_{M}\mathcal{D}^{\dag}-\mathcal{P}^{\dag}*_{N}\mathcal{P}*_{N}\mathcal{Y}^{0}*_{M}\mathcal{D}*_{M}\mathcal{D}^{\dag}
\nonumber\\&-(\mathcal{L_{P}})*_{N}\mathcal{Y}^{0}*_{M}\mathcal{D}*_{M}\mathcal{D}^{\dag}*_{M}\mathcal{Q}*_{M}\mathcal{Q}^{\dag}
+\mathcal{Y}^{0}-\mathcal{Y}^{0}*_{M}\mathcal{D}*_{M}\mathcal{D}^{\dag}\nonumber\\
=&\mathcal{P}^{\dag}*_{N}\mathcal{E}*_{M}\mathcal{D}^{\dag}+
\mathcal{S}^{\dag}*_{N}\mathcal{S}*_{N}\mathcal{C}^{\dag}*_{N}\mathcal{E}*_{M}\mathcal{Q}^{\dag}
-\mathcal{P}^{\dag}*_{N}\mathcal{P}*_{N}\mathcal{Y}^{0}*_{M}\mathcal{D}*_{M}\mathcal{D}^{\dag}
\nonumber\\
&-\mathcal{S}^{\dag}*_{N}\mathcal{S}*_{N}\mathcal{Y}^{0}*_{M}\mathcal{Q}*_{M}\mathcal{Q}^{\dag}+
\mathcal{Y}^{0}\nonumber\\
=&\mathcal{P}^{\dag}*_{N}\mathcal{E}*_{M}\mathcal{D}^{\dag}+
\mathcal{S}^{\dag}*_{N}\mathcal{S}*_{N}\mathcal{C}^{\dag}*_{N}\mathcal{E}*_{M}\mathcal{Q}^{\dag}
-\mathcal{P}^{\dag}*_{N}(\mathcal{R_{A}})*_{N}\mathcal{E}*_{M}\mathcal{D}^{\dag}
\nonumber\\
&-\mathcal{S}^{\dag}*_{N}\mathcal{S}*_{N}\mathcal{Y}^{0}*_{M}\mathcal{Q}*_{M}\mathcal{Q}^{\dag}+
\mathcal{Y}^{0}\nonumber\\
=&\mathcal{P}^{\dag}*_{N}\mathcal{E}*_{M}\mathcal{D}^{\dag}+
\mathcal{S}^{\dag}*_{N}\mathcal{S}*_{N}\mathcal{C}^{\dag}*_{N}\mathcal{E}*_{M}\mathcal{Q}^{\dag}
-\mathcal{P}^{\dag}*_{N}\mathcal{E}*_{M}\mathcal{D}^{\dag}
\nonumber\\
&-\mathcal{S}^{\dag}*_{N}\mathcal{S}*_{N}\mathcal{Y}^{0}*_{M}\mathcal{Q}*_{M}\mathcal{Q}^{\dag}+
\mathcal{Y}^{0}\nonumber\\
=&\mathcal{S}^{\dag}*_{N}\mathcal{S}*_{N}\mathcal{C}^{\dag}*_{N}\mathcal{E}*_{M}\mathcal{Q}^{\dag}
-\mathcal{S}^{\dag}*_{N}\mathcal{S}*_{N}\mathcal{Y}^{0}*_{M}\mathcal{Q}*_{M}\mathcal{Q}^{\dag}+
\mathcal{Y}^{0}.
\end{align}
We now want to prove that
\begin{align}
\mathcal{S}^{\dag}*_{N}\mathcal{S}*_{N}\mathcal{Y}^{0}*_{M}\mathcal{Q}*_{M}\mathcal{Q}^{\dag}
=\mathcal{S}^{\dag}*_{N}\mathcal{S}*_{N}\mathcal{C}^{\dag}*_{N}\mathcal{E}*_{M}\mathcal{Q}^{\dag}.
\end{align}
It follows from (\ref{29equ229}), (\ref{29equ230}), $(\mathcal{R}_{\mathcal{C}})*_{N}\mathcal{E}*_{M}\mathcal{L_{B}}=0$, $\mathcal{S}=\mathcal{C}*_{N}\mathcal{L_{P}}$ and
\begin{align}
\mathcal{B}*_{M}\mathcal{Q}^{\dag}=\mathcal{B}*_{M}\mathcal{B}^{\dag}*_{M}\mathcal{B}*_{M}\mathcal{Q}^{\dag}*_{M}\mathcal{Q}*_{M}\mathcal{Q}^{\dag}=
\mathcal{B}*_{M}(\mathcal{Q}^{\dag}*_{M}\mathcal{Q}*_{M}\mathcal{B}^{\dag}*_{M}\mathcal{B})^{*}*_{M}\mathcal{Q}^{\dag}=0,
\end{align}
that
\begin{align}\label{29equ237}
&\mathcal{S}^{\dag}*_{N}\mathcal{S}*_{N}\mathcal{Y}^{0}*_{M}\mathcal{Q}*_{M}\mathcal{Q}^{\dag}\nonumber\\
=&
\mathcal{S}^{\dag}*_{N}\mathcal{C}*_{N}(\mathcal{L_{P}})*_{N}\mathcal{Y}^{0}*_{M}\mathcal{Q}*_{M}\mathcal{Q}^{\dag}\nonumber\\
=&
\mathcal{S}^{\dag}*_{N}(\mathcal{C}*_{N}\mathcal{Y}^{0}*_{M}\mathcal{Q}*_{M}\mathcal{Q}^{\dag}-
\mathcal{C}*_{N}\mathcal{P}^{\dag}*_{N}\mathcal{P}*_{N}\mathcal{Y}^{0}*_{M}\mathcal{Q}*_{M}\mathcal{Q}^{\dag})\nonumber\\
=&
\mathcal{S}^{\dag}*_{N}[\mathcal{E}*_{M}\mathcal{(L_{B})}*_{M}\mathcal{Q}^{\dag}-
\mathcal{C}*_{N}\mathcal{P}^{\dag}*_{N}(\mathcal{R_{A}})*_{N}\mathcal{E}*_{M}(\mathcal{L_{B}})*_{M}\mathcal{Q}^{\dag}]\nonumber\\
=&\mathcal{S}^{\dag}*_{N}[\mathcal{E}*_{M}\mathcal{(L_{B})}-
\mathcal{C}*_{N}\mathcal{P}^{\dag}*_{N}(\mathcal{R_{A}})*_{N}\mathcal{E}*_{M}(\mathcal{L_{B}})]*_{M}\mathcal{Q}^{\dag}\nonumber\\
=&\mathcal{S}^{\dag}*_{N}[\mathcal{E}*_{M}\mathcal{(L_{B})}-
\mathcal{C}*_{N}\mathcal{P}^{\dag}*_{N}(\mathcal{R_{A}})*_{N}\mathcal{C}*_{N}\mathcal{C}^{\dag}*_{N}\mathcal{E}*_{M}(\mathcal{L_{B}})]*_{M}\mathcal{Q}^{\dag}\nonumber\\
=&\mathcal{S}^{\dag}*_{N}[\mathcal{E}*_{M}\mathcal{(L_{B})}-
\mathcal{C}*_{N}\mathcal{P}^{\dag}*_{N}\mathcal{P}*_{N}\mathcal{C}^{\dag}*_{N}\mathcal{E}*_{M}(\mathcal{L_{B}})]*_{M}\mathcal{Q}^{\dag}\nonumber\\
=&\mathcal{S}^{\dag}*_{N}[\mathcal{E}*_{M}\mathcal{(L_{B})}-
\mathcal{C}*_{N}(\mathcal{I}-\mathcal{L_{P}})*_{N}\mathcal{C}^{\dag}*_{N}\mathcal{E}*_{M}(\mathcal{L_{B}})]*_{M}\mathcal{Q}^{\dag}\nonumber\\
=&\mathcal{S}^{\dag}*_{N}[\mathcal{E}*_{M}\mathcal{(L_{B})}-
\mathcal{C}*_{N}\mathcal{C}^{\dag}*_{N}\mathcal{E}*_{M}(\mathcal{L_{B}})+
\mathcal{C}*_{N}(\mathcal{L_{P}})*_{N}\mathcal{C}^{\dag}*_{N}\mathcal{E}*_{M}(\mathcal{L_{B}})]*_{M}\mathcal{Q}^{\dag}\nonumber\\
=&\mathcal{S}^{\dag}*_{N}[\mathcal{E}*_{M}\mathcal{(L_{B})}-
\mathcal{E}*_{M}(\mathcal{L_{B}})+
\mathcal{C}*_{N}(\mathcal{L_{P}})*_{N}\mathcal{C}^{\dag}*_{N}\mathcal{E}*_{M}(\mathcal{L_{B}})]*_{M}\mathcal{Q}^{\dag}\nonumber\\
=&\mathcal{S}^{\dag}*_{N}\mathcal{S}*_{N}\mathcal{C}^{\dag}*_{N}\mathcal{E}*_{M}\mathcal{Q}^{\dag}.
\end{align}
Hence from (\ref{29equ234}) and (\ref{29equ237}), we see that
\begin{align}
\mathcal{Y}=\mathcal{Y}^{0}.
\end{align}

Now we want to show that $\mathcal{X}=\mathcal{X}^{0}$ where
\begin{align}
\mathcal{U}_{2}=\mathcal{Y}^{0}*_{M}\mathcal{D}*_{M}\mathcal{D}^{\dag},~
\mathcal{U}_{4}=\mathcal{X}^{0}*_{M}\mathcal{B}*_{M}\mathcal{B}^{\dag},~\mathcal{U}_{5}=\mathcal{X}^{0}.
\end{align}
First, we prove that
\begin{align}
\mathcal{X}=\mathcal{A}^{\dag}*_{N}(\mathcal{E}-\mathcal{C}*_{N}\mathcal{Y}*_{M}\mathcal{D})*_{M}\mathcal{B}^{\dag}+\mathcal{(L_{A})}*_{N}\mathcal{U}_{4}+
\mathcal{U}_{5*M}\mathcal{R_{B}},
\end{align}where $\mathcal{X}$ and $\mathcal{Y}$ are given in (\ref{25equ023}) and (\ref{25equ024}), respectively.
From (\ref{29equ217}), (\ref{29equ225}), $(\mathcal{R}_{\mathcal{A}})*_{N}\mathcal{E}*_{M}\mathcal{L_{D}}=0$ and $\mathcal{S}=\mathcal{C}*_{N}\mathcal{L_{P}}$, we infer that
\begin{align}
&\mathcal{A}^{\dag}*_{N}(\mathcal{E}-\mathcal{C}*_{N}\mathcal{Y}*_{M}\mathcal{D})*_{M}\mathcal{B}^{\dag}+\mathcal{(L_{A})}*_{N}\mathcal{U}_{4}+
\mathcal{U}_{5}*_{M}\mathcal{R_{B}}\nonumber\\
=&\mathcal{A}^{\dag}*_{N}\mathcal{E}*_{M}\mathcal{B}^{\dag}-\mathcal{A}^{\dag}*_{N}\mathcal{C}*_{N}\mathcal{Y}*_{M}\mathcal{D}*_{M}\mathcal{B}^{\dag}
+\mathcal{(L_{A})}*_{N}\mathcal{U}_{4}+
\mathcal{U}_{5}*_{M}\mathcal{R_{B}}\nonumber\\
=&\mathcal{A}^{\dag}*_{N}\mathcal{E}*_{M}\mathcal{B}^{\dag}-\mathcal{A}^{\dag}*_{N}\mathcal{C}*_{N}\mathcal{Y}*_{M}\mathcal{D}*_{M}\mathcal{B}^{\dag}
+\mathcal{(L_{A})}*_{N}\mathcal{U}_{4}+
\mathcal{U}_{5}*_{M}\mathcal{R_{B}}\nonumber\\
=&\mathcal{A}^{\dag}*_{N}\mathcal{E}*_{M}\mathcal{B}^{\dag}-
\mathcal{A}^{\dag}*_{N}\mathcal{C}*_{N}\mathcal{P}^{\dag}*_{N}\mathcal{E}*_{M}\mathcal{D}^{\dag}*_{M}\mathcal{D}*_{M}\mathcal{B}^{\dag}
\nonumber\\&-\mathcal{A}^{\dag}*_{N}\mathcal{C}*_{N}\mathcal{S}^{\dag}*_{N}\mathcal{S}*_{N}\mathcal{C}^{\dag}*_{N}\mathcal{E}*_{M}\mathcal{Q}^{\dag}*_{M}\mathcal{D}*_{M}\mathcal{B}^{\dag}
\nonumber\\
&-\mathcal{A}^{\dag}*_{N}\mathcal{S}*_{N}\mathcal{U}_{2}*_{M}(\mathcal{R_{Q}})*_{M}\mathcal{D}*_{M}\mathcal{B}^{\dag}
+(\mathcal{L_{A}})*_{N}\mathcal{U}_{4}+\mathcal{U}_{5}*_{M}\mathcal{R_{B}}\nonumber\\
=&\mathcal{A}^{\dag}*_{N}\mathcal{E}*_{M}\mathcal{B}^{\dag}-
\mathcal{A}^{\dag}*_{N}\mathcal{C}*_{N}\mathcal{P}^{\dag}*_{N}(\mathcal{R_{A}}+\mathcal{A}*_{N}\mathcal{A}^{\dag})*_{N}\mathcal{E}*_{M}\mathcal{D}^{\dag}*_{M}\mathcal{D}*_{M}\mathcal{B}^{\dag}
\nonumber\\
&-\mathcal{A}^{\dag}*_{N}\mathcal{C}*_{N}(\mathcal{L_{P}}+\mathcal{P}^{\dag}*_{N}\mathcal{P})*_{N}\mathcal{S}^{\dag}*_{N}\mathcal{S}*_{N}\mathcal{C}^{\dag}*_{N}\mathcal{E}*_{M}\mathcal{Q}^{\dag}*_{M}\mathcal{D}*_{M}\mathcal{B}^{\dag}
\nonumber\\
&-\mathcal{A}^{\dag}*_{N}\mathcal{S}*_{N}\mathcal{U}_{2}*_{M}(\mathcal{R_{Q}})*_{M}\mathcal{D}*_{M}\mathcal{B}^{\dag}
+(\mathcal{L_{A}})*_{N}\mathcal{U}_{4}+\mathcal{U}_{5}*_{M}\mathcal{R_{B}}\nonumber\\
=&\mathcal{A}^{\dag}*_{N}\mathcal{E}*_{M}\mathcal{B}^{\dag}-
\mathcal{A}^{\dag}*_{N}\mathcal{C}*_{N}\mathcal{P}^{\dag}*_{N}(\mathcal{R_{A}})*_{N}\mathcal{E}*_{M}\mathcal{D}^{\dag}*_{M}\mathcal{D}*_{M}\mathcal{B}^{\dag}
\nonumber\\
&-\mathcal{A}^{\dag}*_{N}\mathcal{C}*_{N}(\mathcal{L_{P}})*_{N}\mathcal{S}^{\dag}*_{N}\mathcal{S}*_{N}\mathcal{C}^{\dag}*_{N}\mathcal{E}*_{M}\mathcal{Q}^{\dag}*_{M}\mathcal{D}*_{M}\mathcal{B}^{\dag}
\nonumber\\
&-\mathcal{A}^{\dag}*_{N}\mathcal{S}*_{N}\mathcal{U}_{2}*_{M}(\mathcal{R_{Q}})*_{M}\mathcal{D}*_{M}\mathcal{B}^{\dag}
+(\mathcal{L_{A}})*_{N}\mathcal{U}_{4}+\mathcal{U}_{5}*_{M}\mathcal{R_{B}}\nonumber\\
=&\mathcal{A}^{\dag}*_{N}\mathcal{E}*_{M}\mathcal{B}^{\dag}-
\mathcal{A}^{\dag}*_{N}\mathcal{C}*_{N}\mathcal{P}^{\dag}*_{N}(\mathcal{R_{A}})*_{N}\mathcal{E}*_{M}\mathcal{B}^{\dag}
-\mathcal{A}^{\dag}*_{N}\mathcal{S}*_{N}\mathcal{C}^{\dag}*_{N}\mathcal{E}*_{M}\mathcal{Q}^{\dag}*_{M}\mathcal{D}*_{M}\mathcal{B}^{\dag}
\nonumber\\
&-\mathcal{A}^{\dag}*_{N}\mathcal{S}*_{N}\mathcal{U}_{2}*_{M}(\mathcal{R_{Q}})*_{M}\mathcal{D}*_{M}\mathcal{B}^{\dag}
+(\mathcal{L_{A}})*_{N}\mathcal{U}_{4}+\mathcal{U}_{5}*_{M}\mathcal{R_{B}}\nonumber\\
=&\mathcal{A}^{\dag}*_{N}\mathcal{E}*_{M}\mathcal{B}^{\dag}-
\mathcal{A}^{\dag}*_{N}\mathcal{C}*_{N}\mathcal{P}^{\dag}*_{N}\mathcal{E}*_{M}\mathcal{B}^{\dag}
-\mathcal{A}^{\dag}*_{N}\mathcal{S}*_{N}\mathcal{C}^{\dag}*_{N}\mathcal{E}*_{M}\mathcal{Q}^{\dag}*_{M}\mathcal{D}*_{M}\mathcal{B}^{\dag}
\nonumber\\
&-\mathcal{A}^{\dag}*_{N}\mathcal{S}*_{N}\mathcal{U}_{2}*_{M}(\mathcal{R_{Q}})*_{M}\mathcal{D}*_{M}\mathcal{B}^{\dag}
+(\mathcal{L_{A}})*_{N}\mathcal{U}_{4}+\mathcal{U}_{5}*_{M}\mathcal{R_{B}}\nonumber\\
=&\mathcal{X}.
\end{align}
Hence, we have
\begin{align}
\mathcal{X}=&\mathcal{A}^{\dag}*_{N}(\mathcal{E}-\mathcal{C}*_{N}\mathcal{Y}^{0}*_{M}\mathcal{D})*_{M}\mathcal{B}^{\dag}
+\mathcal{(L_{A})}*_{N}\mathcal{X}^{0}*_{M}\mathcal{B}*_{M}\mathcal{B}^{\dag}+
\mathcal{X}^{0}*_{M}\mathcal{R_{B}}\nonumber\\
=&\mathcal{A}^{\dag}*_{N}\mathcal{A}*_{N}\mathcal{X}^{0}*_{M}\mathcal{B}*_{M}\mathcal{B}^{\dag}
+\mathcal{X}^{0}*_{M}\mathcal{B}*_{M}\mathcal{B}^{\dag}-\mathcal{A}^{\dag}*_{N}\mathcal{A}*_{N}\mathcal{X}^{0}*_{M}\mathcal{B}*_{M}\mathcal{B}^{\dag}
\nonumber\\&+\mathcal{X}^{0}-\mathcal{X}^{0}*_{M}\mathcal{B}*_{M}\mathcal{B}^{\dag}\nonumber\\
=&\mathcal{X}^{0}.
\end{align}
\end{proof}

Now we consider some special cases of the quaternion tensor equation (\ref{tensorequ01}). In Theorem \ref{theorem01}, let $\mathcal{C}$ and $\mathcal{D}$ vanish. Then we can get the general solution to the following quaternion tensor equation
\begin{align}\label{specsys01}
\mathcal{A}*_{N}\mathcal{X}*_{M}\mathcal{B}=\mathcal{E}.
\end{align}

\begin{corollary}\label{coro22}
Let $\mathcal{A}\in\mathbb{H}^{I_{1}\times \cdots \times I_{N}\times J_{1}\times \cdots \times J_{N}},
\mathcal{B}\in\mathbb{H}^{K_{1}\times \cdots \times K_{M}\times L_{1}\times \cdots \times L_{M}}$ \\
and
$\mathcal{E}\in\mathbb{H}^{I_{1}\times \cdots \times I_{N}\times L_{1}\times \cdots \times L_{M}}$. Then the quaternion tensor equation (\ref{specsys01}) is consistent if and only if
\begin{align*}
(\mathcal{R}_{\mathcal{A}})*_{N}\mathcal{E}=0,\qquad
\mathcal{E}*_{M}(\mathcal{L}_{\mathcal{B}})=0.
\end{align*}
In this case, the general solution to (\ref{specsys01}) can be
expressed as
\begin{align*}
\mathcal{X}=&\mathcal{A}^{\dag}*_{N}\mathcal{E}*_{M}\mathcal{B}^{\dag}
+(\mathcal{L_{A}})*_{N}\mathcal{U}+\mathcal{V}*_{M}\mathcal{R_{B}},
\end{align*}
where $\mathcal{U}$ and $\mathcal{V}$ are arbitrary quaternion tensors with suitable order.
\end{corollary}

Let $\mathcal{B}=\mathcal{I}$ and $\mathcal{C}=\mathcal{I}$ in Theorem \ref{theorem01}. Then we can solve the following generalized Sylvester quaternion tensor equation
\begin{align}\label{specsys02}
\mathcal{A}*_{N}\mathcal{X}+\mathcal{Y}*_{M}\mathcal{D}=\mathcal{E}.
\end{align}

\begin{corollary}\label{coro23}
Let $\mathcal{A}\in\mathbb{H}^{I_{1}\times \cdots \times I_{N}\times J_{1}\times \cdots \times J_{N}},
\mathcal{D}\in\mathbb{H}^{H_{1}\times \cdots \times H_{M}\times L_{1}\times \cdots \times L_{M}}$ and\\
$\mathcal{E}\in\mathbb{H}^{I_{1}\times \cdots \times I_{N}\times L_{1}\times \cdots \times L_{M}}$.
Then the generalized Sylvester quaternion tensor equation (\ref{specsys02}) is consistent if and only if
\begin{align*}
(\mathcal{R}_{\mathcal{A}})*_{N}\mathcal{E}*_{M}\mathcal{L_{D}}=0.
\end{align*}
In this case, the general solution to (\ref{specsys02}) can be
expressed as
\begin{align*}
\mathcal{X}=\mathcal{A}^{\dag}*_{N}\mathcal{E}-\mathcal{U}_{1}*_{M}\mathcal{D}
+(\mathcal{L_{A}})*_{N}\mathcal{U}_{2},
\end{align*}
\begin{align*}
\mathcal{Y}=&(\mathcal{R_{A}})*_{N}\mathcal{E}*_{M}\mathcal{D}^{\dag}+\mathcal{A}*_{N}\mathcal{U}_{1}+
\mathcal{U}_{3}*_{M}\mathcal{R_{D}},
\end{align*}
where $\mathcal{U}_{1},\mathcal{U}_{2},\mathcal{U}_{3}$ are arbitrary quaternion tensors with suitable order.
\end{corollary}

Now we consider the quaternion tensor equation (\ref{htensorequ01}).
Observe that the quaternion tensor equation (\ref{tensorequ01}) plays an important role in investigating the solution to (\ref{htensorequ01}).

\begin{theorem}\label{04theorem54}
Let $\mathcal{A}\in\mathbb{H}^{I_{1}\times \cdots \times I_{N}\times J_{1}\times \cdots \times J_{N}},
\mathcal{C}\in\mathbb{H}^{I_{1}\times \cdots \times I_{N}\times G_{1}\times \cdots \times G_{N}}$  and\\
$\mathcal{E}=\mathcal{E}^{\eta*}\in\mathbb{H}^{I_{1}\times \cdots \times I_{N}\times I_{1}\times \cdots \times I_{N}}$. Set
\begin{align*}
\mathcal{P}=(\mathcal{R_{A}})*_{N}\mathcal{C},\quad  \mathcal{S}=\mathcal{C}*_{N}\mathcal{L_{P}}.
\end{align*}
Then the quaternion tensor equation (\ref{htensorequ01}) has an $\eta$-Hermitian solution if and only if
\begin{align*}
(\mathcal{R}_{\mathcal{P}})*_{N}(\mathcal{R}_{\mathcal{A}})*_{N}\mathcal{E}=0,\qquad
(\mathcal{R}_{\mathcal{A}})*_{N}\mathcal{E}*_{N}\mathcal{R_{C}}=0.
\end{align*}
In this case, the general $\eta$-Hermitian solution to (\ref{htensorequ01}) can be
expressed as
\begin{align*}
\mathcal{X}=\frac{\widehat{\mathcal{X}}+\widehat{\mathcal{X}}^{\eta*}}{2},\qquad
\mathcal{Y}=\frac{\widehat{\mathcal{Y}}+\widehat{\mathcal{Y}}^{\eta*}}{2},
\end{align*}
where
\begin{align*}
\widehat{\mathcal{X}}=&\mathcal{A}^{\dag}*_{N}\mathcal{E}*_{N}(\mathcal{A}^{\dag})^{\eta*}-
\mathcal{A}^{\dag}*_{N}\mathcal{C}*_{N}\mathcal{P}^{\dag}*_{N}\mathcal{E}*_{N}(\mathcal{A}^{\dag})^{\eta*}
\\&-\mathcal{A}^{\dag}*_{N}\mathcal{S}*_{N}\mathcal{C}^{\dag}*_{N}\mathcal{E}*_{N}(\mathcal{P}^{\dag})^{\eta*}*_{N}\mathcal{C}^{*}*_{N}(\mathcal{A}^{\dag})^{\eta*}
\nonumber\\
&-\mathcal{A}^{\dag}*_{N}\mathcal{S}*_{N}\mathcal{U}_{2}*_{M}(\mathcal{L_{P}})*_{N}(\mathcal{C})^{\eta*}*_{N}(\mathcal{A}^{\dag})^{\eta*}
+(\mathcal{L_{A}})*_{N}\mathcal{U}_{4}+\mathcal{U}_{5}*_{N}\mathcal{L_{A}},
\end{align*}
\begin{align*}
\widehat{\mathcal{Y}}=&\mathcal{P}^{\dag}*_{N}\mathcal{E}*_{N}(\mathcal{C}^{\dag})^{\eta*}+
\mathcal{S}^{\dag}*_{N}\mathcal{S}*_{N}\mathcal{C}^{\dag}*_{N}\mathcal{E}*_{N}(\mathcal{P}^{\dag})^{\eta*}
\\&+(\mathcal{L_{P}})*_{N}(\mathcal{L_{S}})*_{N}\mathcal{U}_{1}+(\mathcal{L_{P}})*_{N}\mathcal{U}_{2}*_{N}\mathcal{L_{P}}+
\mathcal{U}_{3}*_{N}\mathcal{L_{C}},
\end{align*}
where $\mathcal{U}_{1},\mathcal{U}_{2},\mathcal{U}_{3},\mathcal{U}_{4},\mathcal{U}_{5}$ are arbitrary quaternion tensors with suitable order.
\end{theorem}

\begin{proof}
We first prove that the quaternion tensor equation (\ref{htensorequ01}) has an $\eta$-Hermitian solution
if and only if the following quaternion tensor equation
\begin{align}\label{30equ041}
\mathcal{A}*_{N}\widehat{\mathcal{X}}*_{N}\mathcal{A}^{\eta*}+\mathcal{C}*_{N}\widehat{\mathcal{Y}}*_{N}\mathcal{C}^{\eta*}=\mathcal{E}
\end{align}
has a solution. If the quaternion tensor equation (\ref{htensorequ01}) has an $\eta$-Hermitian solution,
say, $(\mathcal{X}^{0},\mathcal{Y}^{0})$, then the quaternion tensor equation (\ref{30equ041}) clearly
has a solution $(\widehat{\mathcal{X}},\widehat{\mathcal{Y}})=(\mathcal{X}^{0},\mathcal{Y}^{0})$. Conversely, if the quaternion tensor equation (\ref{30equ041}) has a solution $(\widehat{\mathcal{X}},\widehat{\mathcal{Y}})$, then
\begin{align*}
(\mathcal{X},\mathcal{Y})=\Big(\frac{\widehat{\mathcal{X}}+\widehat{\mathcal{X}}^{\eta*}}{2},\frac{\widehat{\mathcal{Y}}+\widehat{\mathcal{Y}}^{\eta*}}{2}\Big)
\end{align*}
is an $\eta$-Hermitian solution of (\ref{htensorequ01}). Applying Theorem \ref{theorem01}  we can give some solvability conditions and general $\eta$-Hermitian solution to the quaternion tensor equation (\ref{htensorequ01}).
\end{proof}

In Theorem \ref{04theorem54}, let $\mathcal{C}=\mathcal{I}$. Then we can solve the following quaternion tensor equation
\begin{align}\label{htensorequ02}
\mathcal{A}*_{N}\mathcal{X}*_{N}\mathcal{A}^{\eta*}=\mathcal{E},
\quad \mathcal{X}=\mathcal{X}^{\eta*}.
\end{align}

\begin{corollary}
Let $\mathcal{A}\in\mathbb{H}^{I_{1}\times \cdots \times I_{N}\times J_{1}\times \cdots \times J_{N}}$  and
$\mathcal{E}=\mathcal{E}^{\eta*}\in\mathbb{H}^{I_{1}\times \cdots \times I_{N}\times I_{1}\times \cdots \times I_{N}}$.
Then the quaternion tensor equation (\ref{htensorequ02}) has an $\eta$-Hermitian solution if and only if
\begin{align*}
(\mathcal{R}_{\mathcal{A}})*_{N}\mathcal{E}=0.
\end{align*}
In this case, the general $\eta$-Hermitian solution to (\ref{htensorequ02}) can be
expressed as
\begin{align*}
\mathcal{X}=&\mathcal{A}^{\dag}*_{N}\mathcal{E}*_{N}(\mathcal{A}^{\dag})^{\eta*}
+(\mathcal{L_{A}})*_{N}\mathcal{U}+\mathcal{U}^{\eta*}*_{N}(\mathcal{L_{A}})^{\eta*},
\end{align*}
where $\mathcal{U}$ is an arbitrary quaternion tensor with suitable order.
\end{corollary}

\section{\textbf{Conclusion}}

We have derived the quaternion tensor SVD, quaternion tensor rank decomposition, and $\eta$-Hermitian quaternion tensor decomposition with the isomorphic group structures and Einstein product. Using the quaternion tensor SVD, we have given the expression of the Moore-Penrose inverse of a quaternion tensor. We have also established some necessary and sufficient conditions for the existence of the general solution to the generalized Sylvester quaternion tensor equation (\ref{tensorequ01}) in terms of the Moore-Penrose inverses of the quaternion tensors $\mathcal{A},\mathcal{B},\mathcal{C},\mathcal{D},\mathcal{E}$. The expression of the general solution to this tensor equation (\ref{tensorequ01}) has also been given when it is solvable. Finally, we have presented necessary and sufficient conditions for the quaternion tensor equation (\ref{htensorequ01}) to have an $\eta$-Hermitian solution. The expression of the general $\eta$-Hermitian to (\ref{htensorequ01}) has also been provided when its solvability conditions are satisfied.

\end{document}